\documentclass{amsart}

\usepackage[english]{babel}

\usepackage{mathrsfs, mathtools, amssymb}

\usepackage{dsfont}

\usepackage[paper=a4paper, margin=2cm]{geometry}

\usepackage{hyperref}       
\usepackage{url}            
\usepackage{booktabs}       
\usepackage{amsfonts}       
\usepackage{nicefrac}       
\usepackage{microtype}      
\usepackage{lipsum}

\usepackage{amsmath,amssymb,amsthm,amscd}
\usepackage{amsrefs}
\usepackage{tikz,tikz-cd}
\usetikzlibrary{arrows,arrows.meta}

\usepackage{graphicx}
\graphicspath{}
\DeclareGraphicsExtensions{.pdf,.png,.jpg,.jpeg}

\newcounter{stmcounter}[section]

\newcounter{thmMaincounter}

\newtheorem{formula}{}[section]
\newtheorem{proposition}[formula]{Proposition}
\newtheorem{corollary}[formula]{Corollary}
\newtheorem{lemma}[formula]{Lemma}
\newtheorem{theorem}[formula]{Theorem}
\newtheorem{theoremM}[thmMaincounter]{Theorem}

\theoremstyle{definition}

\newtheorem{definition}[formula]{Definition}
\newtheorem{example}[formula]{Example}
\newtheorem{conjecture}[formula]{Conjecture}
\newtheorem{problem}[formula]{Problem}
\newtheorem{condition}[formula]{Condition}
\newtheorem*{condition*}{Condition}
\newtheorem{construction}[formula]{Construction}

\theoremstyle{remark}

\newtheorem*{remark}{Remark}

\newcommand{\zp}{\mathcal Z_P}

\newcommand{\zk}{\mathcal Z_K}

\newcommand{\hrk}{\mathrm{hrk}}
\newcommand{\trk}{\mathrm{trk}}
\newcommand{\Z}{\mathbb Z}
\newcommand{\R}{\mathbb R}
\newcommand{\C}{\mathbb C}
\newcommand{\Q}{\mathbb Q}

\newcommand{\F}{\mathbb F}

\newcommand{\lb}{\lbrace}
\newcommand{\rb}{\rbrace}

\DeclareMathOperator{\rk}{rk}

\DeclareMathOperator{\pt}{pt}

\DeclareMathOperator{\Tor}{Tor}

\DeclareMathOperator{\Id}{Id}
\DeclareMathOperator{\col}{colim}
\DeclareMathOperator{\hoc}{hocolim}
\DeclareMathOperator{\cat}{cat}
\DeclareMathOperator{\Hm}{Hom}
\DeclareMathOperator{\Mat}{Mat}
\DeclareMathOperator{\cone}{cone}

\DeclareMathOperator{\Top}{Top}

\DeclareMathOperator{\Inc}{Inc}
\DeclareMathOperator{\dga}{dga}

\renewcommand{\leq}{\leqslant}
\renewcommand{\geq}{\geqslant}

\tikzcdset{arrow style=tikz, diagrams={>=stealth}}

\begin{document}

\title[Decomposition for quotients of moment-angle complexes and its applications]{On the homotopy decomposition for the quotient of a moment-angle complex and its applications}

\author{Ivan Limonchenko}
\address[I.\,Limonchenko]{National Research University Higher School of Economics, Russian Federation}
\email{ilimonchenko@hse.ru}
\author{Grigory Solomadin}
\address[G.\,Solomadin]{National Research University Higher School of Economics, Russian Federation}
\email{grigory.solomadin@gmail.com}

\subjclass[2020]{57S12, 13F55, 55N91}
\keywords{Homotopy colimit, toric diagram, moment-angle complex, quasitoric manifold, partial quotient, Buchstaber number}

\begin{abstract}

In this paper we prove that the quotient of any real or complex moment-angle complex by any closed subgroup in the naturally acting compact torus on it is equivariantly homotopy equivalent to the homotopy colimit of a certain toric diagram. For any quotient we prove an equivariant homeomorphism generalizing the well-known Davis-Januszkiewicz construction for quasitoric manifolds and small covers. We deduce formality of the corresponding Borel construction space under the natural assumption on the group action in the complex case leading to the new description of the equivariant cohomology for the quotients by any coordinate subgroups. We prove the weak Toral Rank Conjecture for any partial quotient by the diagonal circle action. We give an explicit construction of partial quotients by circle actions having arbitrary torsion in integral cohomology.

\end{abstract}

\maketitle

\section{Introduction}

Geometry and topology of moment-angle complexes and manifolds, and of quasitoric manifolds and small covers introduced in the seminal paper~\cite{da-ja-91} is one of key points of study in toric topology. All moment-angle complexes on one side, and on other side, all quasitoric manifolds and small covers belong to the class of partial quotients being a wide family of topological spaces arising in toric topology. The term ``partial quotient'' was introduced in~\cite{bu-pa-02} for the quotient space of the complex moment-angle complex $\zk=(D^2,S^1)^K$ by any freely acting \textit{subtorus} in $T^m=(S^1)^m$, where $K$ is a simplicial complex on the vertex set $[m]=\{1,2,\ldots,m\}$. Note that in~\cite{fr-21} the quotient of the complex moment-angle complex $\zk$ by an arbitrary closed \textit{subgroup} (that is, a quasitorus) in $T^m$ acting freely on $\zk$ was called a partial quotient. In this paper a \textit{partial quotient} of the (real or complex, respectively) moment-angle complex $(D^d,S^{d-1})^K$, $d=1,2$, is the corresponding quotient by the action of any freely acting closed subgroup $H_d$ from $(G_d)^m$, where $G_1:=\Z/2\Z$ and $G_2:=S^1$ (see \cite{da-ja-91}).

The notion of a \textit{polyhedral product} introduced in \cite{BBCG} is an instance of a colimit for a certain diagram of topological spaces over a small category $\cat K$. The categorical approach to polyhedral products \cite{pa-ra-08} includes the homotopy equivalence of any moment-angle complex \cite{no-ra-05} and of any quasitoric manifold \cite{we-zi-zi-99} to the homotopy colimits of toric diagrams in the terminology of \cite{we-zi-zi-99}, as well as of general partial quotients \cite{fr-10}. This elegant approach has several applications. For example, it implies that any quasitoric manifold is a rationally formal space~\cite{pa-ra-08} and that any complex Davis-Januszkiewicz space is formal \cite{no-ra-05}.

Moment-angle complexes and quasitoric manifolds have already found numerous valuable applications in homotopy theory~\cite{BaskM,DS,L2016,L2017,L2019,BL,GL1,GL2}, cobordism theory~\cite{BPR,LLP,CLP}, hyperbolic geometry~\cite{BEMPP}, combinatorial commutative algebra~\cite{BGLV,LP}. Unlike these two particularly important families of partial quotients, geometry and topology of general partial quotients is still far from being well-understood. Several authors attacked the problem of describing cohomology rings of general partial quotients~\cite{bu-pa-99,Fr06}. However, a complete and rigorous argument giving the multiplicative structure in the cohomology ring $H^*(\zk/H;\ \Z)$ of any partial quotient was given only recently in~\cite{fr-21}. In addition, in~\cite{li-19}, \cite{fu-19} another two classes of quotients of moment-angle complexes were introduced.

In this paper we introduce a new class of pairs $(K,H_d)$, where $K$ is any simplicial complex on $m$ vertices and $H_d$ is any closed subgroup in $G_d^m$ such that the following condition holds (see Condition \ref{cond:1}).

\begin{condition*}
For any $I\subseteq J\in\cat K$ the subgroup $H_d\cap G_d^J$ maps to the subgroup $H_d\cap G_d^I$ under the natural projection $G_d^J\to G_d^I$. Or equivalently, in the following diagram there exists an upper horizontal arrow making it a commutative diagram
\[
\begin{tikzcd}
G_d^J\cap H_d \arrow{r}\arrow[hook]{d} & G_d^I \cap H_d \arrow[hook]{d}\\
G_d^J \arrow{r}  & G_d^I,
\end{tikzcd}
\]
where the lower horizontal arrow is the natural projection.
\end{condition*}

The proposed class contains all partial quotients and in addition has some quotients by non-free actions on a moment-angle complex. The first main result of this paper is given as follows (for the precise definitions see \S\ref{sec:prep}).

\begin{theoremM}\label{thmM:cohstab}
Suppose that $K$ and $H_d$ satisfy the above condition. Then one has
\[
\tilde{H}^{i}(\col\  D)=\lim\  \tilde{H}^{i}(D),\
{\lim}^j\  \tilde{H}^{i}(D)=0,\ j>0,
\]
where $D=S_d, BS_d$. In particular, $\tilde{H}^{odd}(\col\ BS_2;\ \Z)=0$.
\end{theoremM}

Notice that Theorem \ref{thmM:cohstab} generalizes the description of the Stanley-Reisner ring $\Z[K]$ as a limit. Theorem \ref{thmM:cohstab} computes the equivariant cohomology ring of a quotient, see Theorem \ref{thmM:eqcoh} and definition of $BS_d$ below. The second main result of this paper for complex  moment-angle complexes (that is, $d=2$ holds, and the corresponding index is omitted in the next theorem) is given as follows (see Theorem \ref{thm:em}).

\begin{theoremM}\label{thmM:em}
Suppose that the above condition holds for the pair $(K, H)$. Then the Eilenberg-Moore spectral sequence for the fiber inclusion to the Borel construction of the $L$-action on $\mathcal{Z}_{K}/H$ is isomorphic to
\[
\Tor^{i,j}_{H^*(BL)}(\lim H^*(BS);\Z)\Rightarrow H^{i+j} (\mathcal{Z}_K /H),
\]
where $L:=T^m/H$. It collapses at the second page. In particular, the associated graded algebra of $H^{*}(\mathcal{Z}_{K}/H)$ is isomorphic to $\Tor^{*}_{H^*(BL)}(\lim H^*(BS);\Z)$.
\end{theoremM}

The proof of Theorem \ref{thmM:em} generalizes the known proof in the case of partial quotients (for example, see \cite{no-ra-05}). The proofs of Theorems \ref{thmM:cohstab}, \ref{thmM:em} use the following structure theorems that are interesting on their own. (See Theorems \ref{thm:equivquothocolim} and \ref{thm:eqcoh}, respectively.)

\begin{theoremM}\label{thmM:quothocolim}
For any $d=1,2$, any closed subgroup $H_d$ in $G_d^m$ and any simplicial complex $K$ on $[m]$ there is the $L_d$-equivariant homotopy equivalence of spaces
\[
(D^d,S^{d-1})^{K}/H_d\simeq \hoc\ G_d^m/( G_d^I\cdot H_d),
\]
where $L_d:= G_d^m/H_d$, $H_d\cdot G_d^{I}$ is the subgroup generated by $G_d^I:=\prod_{i\in I} G_{d}$ and $H_d$ in $G_d^m$.
\end{theoremM}

\begin{theoremM}\label{thmM:eqcoh}
There is the following homotopy equivalence
\[
EL_d\times_{L_d} (D^d,S^{d-1})^K/H_d\simeq \col B(G_d^I/(G_d^I\cap H_d)),
\]
for the Borel construction of the $L_d$-action on the quotient $(D^d,S^{d-1})^K/H_d$ of the moment-angle complex.
\end{theoremM}

The proof of Theorem \ref{thmM:quothocolim} is given in \S\ref{sec:prep}. Theorem \ref{thmM:quothocolim} was proved in a series of particular cases in~\cite{we-zi-zi-99}, \cite{pa-ra-vo-04}, \cite{fr-10}. In the particular case of a partial quotient the homotopy equivalence from Theorem \ref{thmM:quothocolim} gives rise to an equivariant homeomorphism, where the standard realization of a homotopy colimit is used (Corollary \ref{cor:djcon}). This result leads to the explicit $L_d$-CW-approximation for quotients of moment-angle complexes (Proposition \ref{pr:quotdiag}), where $L_d:=G_d^m/H_d$. The last result generalizes (see Corollary \ref{cor:djcon}) the well-known Davis-Januszkiewicz construction \cite{da-ja-91} to the case of arbitrary quotients. We also indicate the closely related general constructions for partial quotients from \cite{fr-10}, \cite{ba-be-co-hi-17}.

In order to prove Theorem \ref{thmM:eqcoh} we construct the homotopy equivalence between the diagram for the Borel construction of the $L_d$-action on $(D^d,S^{d-1})^K/H_d$ and $B(G_d^I/(G_d^I\cap H_d))$. For $d=2$ we prove formality of the respective Borel construction (Theorem \ref{thm:form}). Our proof uses a similar argument to \cite{no-ra-05}. The third and final main result of this paper is as follows.

\begin{theoremM}\label{thmM:torsion}
Let $G$ be any finitely generated Abelian group. Then there exist a simple polytope $P\subseteq\R^n$ with $m$ facets and a one-dimensional subtorus (a circle) $H\subseteq T^m$ such that $H$ acts freely on the moment-angle manifold $\zp$ and $H^*(\zp/H)$ contains $G$ as a direct summand.
\end{theoremM}

We prove Theorem \ref{thmM:torsion} by using the Hochster type formula from~\cite{li-19}. Furthermore, we show that the weak Toral Rank Conjecture holds for the class of partial quotients for moment-angle complexes by the action of the diagonal circle action. We finish the paper by a list of some related open problems.

\section{Homotopy decomposition for quotients of moment-angle complexes}\label{sec:prep}

Unless explicitly stated otherwise, in this paper the cohomology groups of a topological space are the singular cohomology groups with integral coefficients. Given a simplicial complex $K$ on the vertex set $[m]:=\lb 1,2,\dots,m\rb$, the objects of $K$ together with the initial object (an empty set) form a small category $\cat K$ with the arrows induced by the natural inclusions of subsets from $[m]$. The formula $I\in\cat K$ means that $I$ is an object of $\cat K$, that is, either $I=\varnothing$ or $I\in K$ holds.

\subsection{Algebraic preparation and definitions of some diagrams}

The proof of the following lemma is straight-forward.

\begin{lemma}\label{lm:stupid}
Consider the following commutative diagram of abelian group homomorphisms, where each of $a$, $a'$, $c$, $c'$ is mono:
\[
\begin{tikzcd}[sep=.3cm]
& A' \arrow{rr}{a'}\arrow{dd} && B' \arrow{dd}\\
 A\arrow[crossing over]{rr}{\qquad a}\arrow{dd}\arrow{ru} && B\arrow{dd}\arrow{ru} &\\
& C' \arrow{rr}{\qquad c'} && D'.\\
C \arrow{rr}{c}\arrow{ru} && D\arrow[crossing over, from=uu]\arrow{ru} &
\end{tikzcd}
\]
Then there is the following commutative diagram of group homomorphisms, where any row is exact:
\[
\begin{tikzcd}[sep=.3cm]
& 1\arrow{rr} && A'  \arrow{rr}\arrow{dd} && B' \arrow{dd}\arrow{rr} && B'/A'\arrow{dd}\arrow{rr} && 1\\
1\arrow{rr} && A\arrow[crossing over]{rr}\arrow{dd}\arrow{ru} && B\arrow{dd}\arrow{ru}\arrow[crossing over]{rr} && B/A\arrow{dd}\arrow{ru}\arrow[crossing over]{rr}&& 1 &\\
& 1\arrow{rr} && C' \arrow{rr} && D'\arrow{rr} && D'/C' \arrow{rr} && 1.\\
1\arrow{rr} && C \arrow[crossing over, from=uu]\arrow{rr}\arrow{ru} && D\arrow[crossing over, from=uu]\arrow{ru}\arrow{rr} && D/C\arrow[crossing over, from=uu]\arrow{ru}\arrow{rr} && 1 &
\end{tikzcd}
\]
\end{lemma}

Following \cite{da-ja-91}, we make use of the notation:
\[
G_{d}:=
\begin{cases}
\Z/2\Z,\ d=1,\\
S^1,\ d=2,
\end{cases}
\F_{d}:=
\begin{cases}
\R,\ d=1,\\
\C,\ d=2,
\end{cases}
R_{d}:=
\begin{cases}
\Z/2\Z,\ d=1,\\
\Z,\ d=2.
\end{cases}
\]
We call any subgroup isomorphic to $G_d^m$ for some $m\geq 0$ a \textit{real torus} for $d=1$, and a \textit{complex torus} for $d=2$, respectively. We call both a real and a complex torus a \textit{torus}.

Throughout the paper we use the standard formalism of polyhedral products, see \cite{bu-pa-15}. Let $H_d$ be any closed subgroup in $G_d^m$. By a slight abuse of the notation we identify the group $G_d^I:=\prod_{i\in I} G_{d}$ with the isomorphic coordinate subgroup $(G_d,1)^I:=\prod_{i\in I} (G_{d})_i\times \prod_{j\in[m]\setminus I} 1_j$ in $G_d^m$, $I\subseteq [m]$. In particular, by the definition one has $G_d^{\varnothing}:=\lb 1\rb$. Denote by $\varphi_{I}\colon G_d^m\to G_d^m/G_d^I$ the natural quotient epimorphism of groups for any $I\subseteq [m]$.

\begin{proposition}\label{pr:bigdiag}
For any $I\subseteq J\in\cat K$ the following diagram
\begin{equation}\label{eq:maincommdiag}
\begin{tikzcd}
1 \arrow{r} & G_d^{I}/(G_d^{I}\cap H_d) \arrow{r}\arrow{d} & G_d^m / H_d \arrow{r}\arrow[equal]{d} & G_d^{m}/(H_d\cdot G_d^{I}) \arrow{r}\arrow{d} & 1\\
1 \arrow{r} & G_d^{J}/(G_d^{J}\cap H_d) \arrow{r} & G_d^m / H_d \arrow{r} & G_d^{m}/(H_d\cdot G_d^{J}) \arrow{r} & 1,
\end{tikzcd}
\end{equation}
is commutative and has exact rows, where $H_d\cdot G_d^{I}$ denotes the subgroup in $G_d^m$ generated by $H_d$ and $G_d^I$.
\end{proposition}
\begin{proof}
There is the following short exact sequence of groups
\begin{equation}\label{eq:simpquot}
\begin{tikzcd}[sep=.9cm]
1 \arrow{r} & G_d^I\cap H_d \arrow{r} & H_d  \arrow{r}{\varphi_{I}|_{H_d}} & \varphi_{I}(H_d) \arrow{r} & 1,
\end{tikzcd}
\end{equation}
so that $\varphi_{I}(H_d)\subseteq G^m_d/G^{I}_d$ holds. Clearly, \eqref{eq:simpquot} is functorial with respect to $I\in\cat K$ (that is, the upper face of the diagram \eqref{eq:bigdiag} below is commutative). One has an isomorphism
\[
(G_d^m/G_d^{I})/\varphi_{I}(H_d)\cong G_d^{m}/(H_d\cdot G_d^{I}).
\]
Then one constructs three out of four cubes of the following commutative diagram with exact rows and columns by applying Lemma \ref{lm:stupid} to the upper left cube:
\begin{equation}\label{eq:bigdiag}
\begin{tikzcd}[sep=.05cm]
    & & & 1 \arrow{dd} & & 1 \arrow{dd} & & 1 \arrow{dd} & & \\
& & 1 & & 1 & & 1 & & &\\
    & 1 \arrow{rr} & & G_d^{J}\cap H_d \arrow{rr}\arrow{dd} & & H_d \arrow{rr}\arrow{dd} & & \varphi_{J} (H_d) \arrow{rr}\arrow{dd} & & 1\\
1 \arrow{rr} & & G_d^{I}\cap H_d \arrow{ru}\arrow[from=uu, crossing over]\arrow[crossing over]{rr} & & H_d \arrow[equal]{ru}\arrow[from=uu, crossing over]\arrow[crossing over]{rr} & & \varphi_{I} (H_d) \arrow{ru}\arrow[from=uu, crossing over]\arrow[crossing over]{rr} & & 1 &\\
    & 1 \arrow{rr} & & G_d^{J} \arrow{rr}\arrow{dd} & & G_d^m \arrow{rr}\arrow{dd} & & G_d^m/G_d^{J} \arrow{rr}\arrow{dd} & & 1\\
1 \arrow{rr} & & G_d^{I} \arrow{ru}\arrow[from=uu, crossing over]\arrow[crossing over]{rr} & & G_d^m \arrow[equal]{ru}\arrow[from=uu, crossing over]\arrow[crossing over]{rr} & & G_d^m/G_d^{I} \arrow{ru}\arrow[from=uu, crossing over]\arrow[crossing over]{rr} & & 1 &\\
    & 1 \arrow{rr} & & G_d^{J}/(G_d^{J}\cap H_d) \arrow{rr}\arrow{dd} & & G_d^m / H_d \arrow{rr}\arrow{dd} & & G_d^m/(H_d\cdot G_d^{J}) \arrow{rr}\arrow{dd} & & 1.\\
1 \arrow{rr} & & G_d^{I}/(G_d^{I}\cap H_d) \arrow{ru}\arrow[from=uu, crossing over]\arrow[crossing over]{rr} & & G_d^m / H_d \arrow[equal]{ru}\arrow[from=uu, crossing over]\arrow[crossing over]{rr} & & G_d^m/(H_d\cdot G_d^{I}) \arrow{ru}\arrow[from=uu, crossing over]\arrow[crossing over]{rr} & & 1 &\\
    & & & 1 & & 1 & & 1 & & \\
& & 1 \arrow[from=uu, crossing over] & & 1 \arrow[from=uu, crossing over] & & 1 \arrow[from=uu, crossing over] & & &
\end{tikzcd}
\end{equation}
There are two different ways to define the right bottom cube of this diagram by applying Lemma \ref{lm:stupid}. A simple check verifies that these two cubes coincide. So the above diagram is well-defined and commutative. The necessary diagram \eqref{eq:maincommdiag} is then given by restricting \eqref{eq:bigdiag} to the lowest face.
\end{proof}

\begin{definition}\label{defn:diags}
Define $S_d$, $\kappa (G_d^{m}/H_d)$, $Q_d$ to be the $(\cat K)$-diagrams of topological spaces such that the arrow corresponding to $I\subseteq J\in \cat K$ is given by left, central and right-most columns of the diagram \eqref{eq:maincommdiag}, respectively, where $\kappa(G_d^{m}/H_d)$ is the constant $(\cat K)$-diagram corresponding to $G_d^{m}/H_d$.
\end{definition}

In what follows we write $Q_d=Q_d(K,H)$ to indicate dependency of the diagram $Q_d$ on $K$ and $H_d$ and use the similar notation for $S_d$ and $\kappa (G_d^{m}/H_d)$ if it is necessary.

\begin{remark}
Let $D\in \Top^C$ be any diagram over a small category $C$ with values in the category $\Top$ of topological spaces. Suppose that any object of $D$ is a torus and any its arrow is a group homomorphism. Then $D$ is called a \textit{toric diagram} (\cite{we-zi-zi-99}). The diagrams $S_d$, $\kappa (G_d^{m}/H_d)$, $Q_d$ are toric diagrams.
\end{remark}

\begin{corollary}\label{cor:seqdiag}
There is the following sequence of $(\cat K)$-diagram morphisms
\begin{equation}\label{eq:chaindiag}
\begin{tikzcd}
\kappa(1) \arrow{r} & S_d \arrow{r} & \kappa(G_d^{m}/H_d) \arrow{r} & Q_d \arrow{r} & \kappa(1),
\end{tikzcd}
\end{equation}
given by \eqref{eq:maincommdiag}. Objectwise \eqref{eq:chaindiag} is a short exact sequence of groups. The diagram $S_d$ is cofibrant.
\end{corollary}
\begin{proof}
The homomorphism $S_d(I\to J)\colon G_d^{I}/(G_d^{I}\cap H_d) \to G_d^{J}/(G_d^{J}\cap H_d)$ has a trivial kernel for any $I\subset J\in K$, so $S_d$ is cofibrant by \cite[Lemma 4.10, p.134]{we-zi-zi-99}. The remaining claims are clear.
\end{proof}

We will make use of the next basic properties of tori.

\begin{proposition}\label{pr:tori}
$(i)$ Any closed subgroup $H_d$ of the torus $G_d^m$ is a quasitorus, that is, it is isomorphic to the direct product of a finite abelian group and of a compact complex torus;

$(ii)$ For any closed subgroup $H_d$ of the torus $G_d^m$ the natural exact sequence of groups
\[
\begin{tikzcd}
1\arrow{r} & H_d \arrow{r} & G_d^m \arrow{r} & G_d^m /H_d \arrow{r} & 1,
\end{tikzcd}
\]
splits iff $d=1$ or $H_d$ is connected;

$(iii)$ Quotient of a torus by any its closed subgroup is isomorphic to a torus.
\end{proposition}
\begin{proof}
The claim $(i)$ is given in \cite[pp. 114]{vi-on-88}. If $d=1$, then the exact sequence from $(ii)$ splits as a sequence of finite-dimensional linear spaces over $\Z/2\Z$. If $d=2$ and the group $H_d$ is connected, then there exists a subtorus $T$ in $G_d^m$ such that the equality $G_d^m = H_d\times T$ holds. The natural projection to the second factor coincides with the quotient homomorphism $G_d^m \to G_d^m /H_d$. A non-canonical section $T \to G_d^m$ of this projection gives a splitting of the exact sequence $(ii)$. If $d=2$ and the group $H_d$ is not connected, then the connected group $G_d^m$ cannot be represented as a direct product of a non-connected group $H_d$ and of a group $G_d^m /H_d$. This proves the claim~$(ii)$. Recall that the image of a closed abelian Lie group with respect to an epimorphism is connected and abelian. This implies the claim $(iii)$ directly, because the quotient homomorphism is epimorphic and closed.
\end{proof}

\subsection{Homotopy decomposition for $(D^d,S^{d-1})^{K}/H_d$ and subgroup arrangements}

In what follows, throughout the paper we use the standard formalism of homotopy colimits for diagrams with values in (pointed) topological spaces. We refer to the sources \cite{bo-ka-72}, \cite{dw-sp-95}, \cite{we-zi-zi-99} for the foundations of the corresponding theory. Unless explicitly stated otherwise, throughout the paper we consider limits and colimits (as well as the corresponding homotopy analogues) over small categories $\cat^{op} K$ and $\cat K$ and with values in the category of compactly generated Hausdorff topological spaces $Top$ only, respectively. Often in the text below we refer to a $(\cat K)$- or $\cat^{op} K$-diagram $D$ by referring to its objects $D(I)$ (as a function on $I$) for brevity.

Let $K$ be a simplicial complex on $[m]$. Recall that there is the $(\cat K)$-diagram $(D^d,S^{d-1})^I$ of topological spaces \cite{bu-pa-15} given by the maps
\[
(D^d,S^{d-1})^I\to (D^d,S^{d-1})^J,\ (D^d,S^{d-1})^I:=\prod_{i\in I} D_i^d \times \prod_{j\in [m]\setminus I} S^{d-1}_{j}\subseteq \prod_{i=1}^{m} D^{d}_{i},
\]
induced by the identity map $\Id\colon D^d\to D^d$ and by the embedding of the boundary map $S^{d-1}=\partial D^d\to D^d$, where $I\subseteq J\in \cat K$. The respective colimits
\[
\R\mathcal{Z}_K=(D^1,S^0)^K:=\col\ (D^1,S^0)^I,\ \mathcal{Z}_K=(D^2,S^1)^K:=\col\ (D^2,S^1)^I,
\]
are called the \textit{real and complex moment-angle complex}, respectively (\cite{bu-pa-15}). Furthermore (\cite{bu-pa-15}), the natural quotient homomorphisms
\[
G_d^m/G_d^I\to G_d^m/G_d^J,
\]
for any $I\subseteq J\in \cat K$, form another $(\cat K)$-diagram $G_d^m/G_d^I$ of spaces.

\begin{proposition}{(\cite{pa-ra-vo-04},\cite[Proposition 8.1.5]{bu-pa-15})}\label{pr:coho}
The diagram $(D^d,S^{d-1})^{I}$ is cofibrant. One has the following homotopy equivalence:
\[
(D^d,S^{d-1})^K\simeq \hoc\ G_d^{m}/G_d^I.
\]
\end{proposition}

\begin{example}\label{ex:s3}
Let $m=2$, $d=2$, $K=\lb \lb 1\rb, \lb 2\rb\rb$. Then $\mathcal{Z}_{K}=S^3$ holds. In this case, the homotopy colimit of $G_d^m/G_d^I$ over $\cat K$ is obtained by gluing the boundary components of the cylinder $I^1\times T^2$ to two disjoint copies of $S^1$ by two different coordinate projections from $T^2:=(S^1)^2$ to $S^1$.
\end{example}

Let $H_d$ be any closed subgroup of the torus $G_d^m$. The subset $(D^d,S^{d-1})^I$ of $(D^d)^m$ is $G_d^m$-invariant with respect to the natural $G_d^m$-action on $(D^d)^m$. Hence, the subset $(D^d,S^{d-1})^I$ is $H_d$-invariant in $(D^d)^m$ for any $I\in \cat K$. Then there are the induced embeddings of the orbit spaces
\begin{equation}\label{eq:projspaces}
(D^d,S^{d-1})^I/H_d\to (D^d,S^{d-1})^J/H_d,
\end{equation}
where $I\subseteq J\in K$, forming the $(\cat K)$-diagram $(D^d,S^{d-1})^I/H_d$ of spaces.

\begin{proposition}\label{pr:quotdiag}
The $(\cat K)$-diagram $(D^d,S^{d-1})^I/H_d$ is cofibrant. One has
\[
(D^d,S^{d-1})^K/H_d=\col (D^d,S^{d-1})^I/H_d.
\]
\end{proposition}
\begin{proof}
Any morphism in the diagram $(D^d,S^{d-1})^I/H_d$ is a closed immersion. Hence, the first claim follows by \cite[Lemma 4.10, p.134]{we-zi-zi-99}. The second claim follows from commutation of colimits by representing the quotient by $H_d$-action as the respective colimit.
\end{proof}

\begin{proposition}\label{pr:hocs}
One has the homotopy equivalence
\[
\hoc\ S_d\simeq \bigcup_{I\in K} G_d^{I}/(G_d^I \cap H_d)\subseteq G_d^m/H.
\]
\end{proposition}
\begin{proof}
The claim follows trivially from cofibrancy of $S_d$ by Corollary \ref{cor:seqdiag} due to the Projection Lemma from \cite{we-zi-zi-99}.
\end{proof}

Denote by
\[
\pi_{I}\colon (D^d,S^{d-1})^I\to G_d^{m}/G_d^{I},
\]
the map induced by the identity map $\Id\colon S^{d-1}\to S^{d-1}$ and by the projection $D^d\to 1$. Notice that the equality
\[
\pi_{I}(g x)=\varphi_{I}(g)\pi_{I}(x),
\]
holds for any $g\in G_d^m$ and any $x\in (D^d,S^{d-1})^I$. Hence, $\pi_{I}$ is an equivariant map with respect to the respective $(H_d)$- and $(\varphi_{I}(H_d))$-actions. Therefore, the map $\pi_{I}$ induces the map of the orbit spaces
\[
\widetilde{\pi}_{I}\colon (D^d,S^{d-1})^I/H_d\to G_d^{m}/(G_d^I\cdot H_d).
\]

\begin{proposition}\label{pr:fbindiag}
$(i)$ The restriction
\[
G_d^m/H_d\to G_d^{m}/(G_d^I\cdot H_d),
\]
of $\widetilde{\pi}_{I}$ to $G_d^m/H_d$ is the trivial principal fiber bundle with the fiber $G_d^I / (H_d\cap G_d^{I})$.

$(ii)$ The map $\widetilde{\pi}_{I}$ is the trivial fiber bundle associated with the principal fiber bundle from $(i)$ under the natural action of $G_d^{I}/(H_d\cap G_d^{I})$ on $(D^{d})^I/(H_d\cap G_d^{I})$, where $(D^{d})^I:=(\prod_{i\in I} D^{d}_{i})$. The fiber of $\widetilde{\pi}_{I}$ is equal to $(D^{d})^I/(H_d\cap G_d^{I})$.
\end{proposition}
\begin{proof}
By Proposition \ref{pr:bigdiag}, the following sequence of groups
\begin{equation}\label{eq:sesg}
\begin{tikzcd}
1\arrow{r} & G_d^I /(H_d\cap G_d^{I})\arrow{r} & G_d^m/H_d \arrow{r} & G_d^{m}/( G_d^I\cdot H_d) \arrow{r} & 1
\end{tikzcd},
\end{equation}
is exact. By Proposition \ref{pr:tori} $(ii)$, $(iii)$, every group from the sequence \eqref{eq:sesg} is a torus, and in particular, is a Lie group, and the short exact sequence \eqref{eq:sesg} of groups splits. Hence, there exists a section of the short exact sequence \eqref{eq:sesg}. Therefore, \eqref{eq:sesg} defines a trivial fibre bundle. This proves $(i)$.

Let
\[
G:=G_d^{I}/(H_d\cap G_d^{I}),\ X:= G_d^m/H_d,\ Y:=(D^{d})^I/(H_d\cap G_d^{I}).
\]
The group $G$ acts on $X$ by left translations as its subgroup, see \eqref{eq:sesg}. The natural embedding $G_d^{I}\to (D^{d})^I$ is $G_d^I$-equivariant. Hence, this embedding induces the $G$-action on Y. Recall that the standard action of $G$ on $X\times Y$ given by the formula $g(x,y):=(gx,yg^{-1})$ for $x\in X$, $y\in Y$ and $g\in G$, has the orbit space $X\times_{G} Y:=(X\times Y)/G$.

Denote by $d\times t'$ and $t\times t'$ arbitrary elements of $(D^d,S^{d-1})^I$ and $G_d^m$, respectively, where $d\in (D^{d})^I$; $t\in G_d^I$; $t'\in G_d^{[m]\setminus I}$. Then, for instance, the natural $G_d^m$-action on $(D^d,S^{d-1})^I$ is given by the formula
\[
g\cdot (d\times t')=\bigl( (\varphi_{[m]\setminus I}(g)\cdot d)\times (\varphi_{I}(g)\cdot t')\bigr),\ g\in G_d^m.
\]

Consider the map
\[
\Psi\colon X\times_{G} Y\to (D^d,S^{d-1})^{I}/H_d,
\]
\begin{equation}\label{eq:defpsi}
\biggl[\bigl(\bigl[t\times t']_{H_d},\bigl[d]_{H_d\cap G_d^{I}}\bigr)\biggr]_{G}\mapsto \bigl[(t\cdot d)\times t'\bigr]_{H_d},
\end{equation}
where, for instance, $[t\times t']_{H_d}$ denotes the $H_d$-orbit in $X$ represented by $t\times t'$. Any element from $G_d^m\times (D^d)^I$ representing the same $G$-orbit in $X\times_{G} Y$ as on the left hand side of \eqref{eq:defpsi} has the form $(gh\cdot (t\times t'), r\cdot d\cdot g^{-1})$, where $g\in G_d^I$; $h\in H_d$; $r\in H_d\cap G_d^{I}$. We compute the value of $\Psi$ on this element as follows:
\begin{multline}
\Psi \biggl[\bigl(\bigl[gh\cdot(t\times t')]_{H_d},\bigl[r\cdot d\cdot g^{-1}]_{H_d\cap G_d^{I}}\bigr)\biggr]_{G} =
\biggl[\biggl( \varphi_{[m]\setminus I}(gh) t\cdot \varphi_{[m]\setminus I}(g)^{-1}\cdot r \cdot d\biggr)\times \varphi_{I}(gh)t'\biggr]_{H_d} = \\
\biggl[\bigl( \varphi_{[m]\setminus I}(h) t\cdot r \cdot d\bigr)\times \varphi_{I}(h)t'\biggr]_{H_d} =
\biggl[ rh\cdot \bigl((t\cdot d)\times t'\bigr)\biggr]_{H_d}=
\bigl[(t\cdot d)\times t'\bigr]_{H_d}.
\end{multline}
In the second equality we use the identity $\varphi_{I}(g)=1$ which follows easily from the definition of $\varphi_{I}$. Hence, the map $\Psi$ is well defined.

Consider the map
\[
\Phi\colon (D^d,S^{d-1})^{I}/H_d\to X\times_{G} Y,
\]
\begin{equation}\label{eq:defphi}
[d\times t']_{H_d}\mapsto
\biggl[\bigl(\bigl[1\times t']_{H_d},\bigl[d]_{H_d\cap G_d^{I}}\bigr)\biggr]_{G}.
\end{equation}

Any element from $(D^d,S^{d-1})^I$ representing the same $H_d$-orbit in $(D^d,S^{d-1})^I/H_d$ as on the left hand side of \eqref{eq:defphi} has the form $h\cdot (d\times t')$, where $d\in (D^d)^I$; $h\in H_d$; $t'\in G_d^{[m]\setminus I}$. We compute the value of $\Phi$ on this element as follows:

\begin{multline}
\Phi \bigl[ h\cdot (d\times t')\bigr]_{H_d}=
\biggl[\biggl(\bigl[1\times  \varphi_{I}(h)\cdot t']_{H_d},\bigl[ \varphi_{[m]\setminus I}(h)\cdot d]_{H_d\cap G_d^{I}}\biggr)\biggr]_{G}=\\
\biggl[\biggl(\bigl[\varphi_{[m]\setminus I}(h)\cdot 1\times  \varphi_{I}(h)\cdot t']_{H_d},[ d]_{H_d\cap G_d^{I}}\biggr)\biggr]_{G}=
\biggl[\biggl(\bigl[1\times t']_{H_d},\bigl[d]_{H_d\cap G_d^{I}}\biggr)\biggr]_{G}.
\end{multline}
Hence, the map $\Phi$ is well defined.

We prove that $\Psi$ and $\Phi$ are mutually inverse maps as follows:
\[
\Phi\Psi \biggl[\bigl(\bigl[t\times t']_{H_d},\bigl[d]_{H_d\cap G_d^{I}}\bigr)\biggr]_{G}=
\Phi [t\cdot d\times t']_{H_d}=
\biggl[\bigl(\bigl[1\times t']_{H_d},\bigl[t\cdot d]_{H_d\cap G_d^{I}}\bigr)\biggr]_{G}=
\biggl[\bigl(\bigl[t\times t']_{H},\bigl[d]_{H_d\cap G_d^{I}}\bigr)\biggr]_{G},
\]
\[
\Psi\Phi [d\times t']_{H_d}=
\Psi \biggl[\bigl(\bigl[1\times t']_{H_d},\bigl[d]_{H_d\cap G_d^{I}}\bigr)\biggr]_{G}=
[d\times t']_{H_d}.
\]
Notice that there is the following diagram
\begin{equation}\label{eq:somediag}
\begin{tikzcd}
X\times_{G} Y \arrow{r} \arrow{d}{\Psi} & X/G \arrow{d}\\
(D^d,S^{d-1})^I/H_d \arrow{r} & G_d^{m}/( G_d^I\cdot H_d),
\end{tikzcd}
\end{equation}
where the right vertical arrow is the group isomorphism given by taking the respective quotient in \eqref{eq:sesg}, and the horizontal arrows are the respective projections. The diagram \eqref{eq:somediag} is commutative because both compositions from the formula \eqref{eq:somediag} map the element on the left side of \eqref{eq:defpsi} to $[t']_{H_{d}}$. This implies the claim about the fiber bundle from $(ii)$. The claim about triviality of the associated fiber bundle from $(ii)$ then follows directly from $(i)$.
\end{proof}

\begin{corollary}\label{cor:dirhomeo}
There is the following homeomorphism of spaces:
\[
(D^d,S^{d-1})^I/H_d\cong
\bigl(G_d^{m}/( G_d^I\cdot H_d)\bigr) \times \bigl((D^{d})^I/(H_d\cap G_d^{I})\bigr).
\]
\end{corollary}

\begin{proposition}\label{pr:fibercontr}
For any closed subgroup $H_d$ of $G_d^m$, the orbit space $(D^d)^m/H_d$ is contractible.
\end{proposition}
\begin{proof}
The homotopy given by mapping $x$ to $(1-t) x$, where $x\in (D^{d})^I$, $t\in [0,1]$, is $H_d$-equivariant due to the natural embedding of groups $H_d\subseteq G_d^m\subseteq O(dm)$.  Hence, it induces the deformation retraction of $(D^{d})^m/H_d$ to the point.
\end{proof}

\begin{proposition}
The following diagram
\begin{equation}\label{eq:fibemb}
\begin{tikzcd}
(D^d,S^{d-1})^I/H_d \arrow{r}{f_1} \arrow{d}{\tilde{\pi}_{I}} & (D^d,S^{d-1})^J/H_d \arrow{d}{\tilde{\pi}_{J}}\\
G_d^{m}/( G_d^I\cdot H_d) \arrow{r}{f_2} & G_d^{m}/( G_d^J\cdot H_d),
\end{tikzcd}
\end{equation}
is commutative for any $I\subseteq J\in \cat K$, where $f_1$, $f_2$ denote the arrows from the respective diagrams.
\end{proposition}
\begin{proof}
Let $(d,t',t'')\in (D^d,S^{d-1})^I$, $(d,d',t'')\in (D^d,S^{d-1})^J$, where $d\in (D^d)^I$; $t'\in (S^{d-1})^{J\setminus I}$; $t''\in (S^{d-1})^{[m]\setminus J}$; $d'\in (D^{d})^{J\setminus I}$. We check commutativity of the above diagram as follows.
\[
f_2 \circ \tilde{\pi}_{I} \bigl([d,t',t'']_{H_d}\bigr)=
f_2\bigl([(1,t',t'')]_{G_d^I\cdot H_d}\bigr)=
\bigl[(1,t',t'')\bigr]_{G_d^J\cdot H_d}=
\bigl[(1,1,t'')\bigr]_{G_d^J\cdot H_d},
\]
\[
\tilde{\pi}_{J} \circ f_1 \bigl([d,t',t'']_{H_d}\bigr)=
\tilde{\pi}_{J}\bigl([(d,t',t'')]_{H_d}\bigr)=
\bigl[(1,1,t'')\bigr]_{G_d^J\cdot H_d}.
\]
Hence, \eqref{eq:fibemb} is commutative.
\end{proof}

\begin{corollary}\label{cor:diagmorph}
The maps $\widetilde{\pi}_{I}$, where $I$ runs over $\cat K$, constitute a well-defined morphism\\ $(D^{d},S^{d-1})^I/H_d \to Q_d$ of $(\cat K)$-diagrams.
\end{corollary}

Recall that a partial quotient (see Introduction) is a quotient of the moment-angle complex $(D^d,S^{d-1})^K$ by any closed freely acting subgroup $H_d$ in $G_d^m$. The following theorem was previously known in the case of partial quotients (\cite[Corollary 2.3]{fr-10}, \cite[Proposition 4.1]{fr-21}).

\begin{theorem}\label{thm:quothocolim}
For any $d=1,2$, any closed subgroup $H_d$ in $G_d^m$ and any simplicial complex $K$ on $[m]$ there is the homotopy equivalence of spaces
\[
(D^d,S^{d-1})^{K}/H_d\simeq \hoc\ G_d^m/( G_d^I\cdot H_d),
\]
where $H_d\cdot G_d^{I}$ is the subgroup generated by $G_d^I$ and $H_d$ in $G_d^m$.
\end{theorem}
\begin{proof}
Any arrow from the morphism of diagrams in Corollary \ref{cor:diagmorph} is a fiber bundle projection with contractible fiber onto the respective base by Proposition \ref{pr:fibercontr}. By Proposition \ref{pr:quotdiag}, the Homotopy and Projection Lemmas \cite{we-zi-zi-99} imply that
\[
(D^d,S^{d-1})^{K}/H_d =\col\ (D^d, S^{d-1})^I/H_d \simeq\hoc\ Q_d=\hoc\ G_d^{m}/( H_d\cdot G_d^I),
\]
hold. The proof is complete.
\end{proof}

\begin{example}
For $H_d=1$ Theorem \ref{thm:quothocolim} gives the well-known homotopy colimit description for the moment-angle complex, see Proposition \ref{pr:coho}.
\end{example}

\begin{example}\label{ex:rp3}
Let $d=2$, $m=2$, $K=\lb \lb 1\rb, \lb 2\rb\rb$, and let $H_d=\Z/ 2\Z$ be a group acting on $T^2=(S^1)^2$ by the formula $h(x,y)=(-x,-y)$, where $h$ is the generator of $H_d$. Notice that $T^2/H_d$ is homeomorphic to $T^2$. Then one has $\mathcal{Z}_{K}=S^3$, $\mathcal{Z}_{K}/H_d=\R P^3$. The natural projection map
\[
p_i\colon T^2=S^1 \times S^1 \to S^1
\]
to the $i$-th factor is equivariant with respect to the $H_d$-action on $T^2$ and the action of the group $\Z/2\Z$ on $S^1$ by the involution, where $i=1,2$. Hence, the map of the orbit spaces
\[
q_i\colon T^2/H_d \to S^1/(\Z/2\Z),
\]
is well defined. Identify $T^2/H_d\cong T^2$, $S^1/(\Z/2\Z)\cong S^1$ by Proposition \ref{pr:tori} $(iii)$. Then the homotopy colimit $\hoc\ G_d^{m}/( H_d\cdot G_d^I)$ in the standard realization is equal to
\begin{equation}\label{eq:gluehocex}
S^1\sqcup_{0\times q_1} I^1\times T^2 \sqcup_{0\times q_1} S^1.
\end{equation}
The projection $q_i$ is uniquely defined by the embedding of the corresponding character lattices. We describe the corresponding embedding. Choose the standard basis $e_1, e_2$ of the lattice $\Z^2\cong \Hm(T^2,S^1)$. Next, choose the basis
\[
f_1:=e_1,\ f_2:=e_1+e_2,
\]
in $\Z^2$. The last basis agrees with the splitting of the torus $T^2$ into the direct product of the diagonal circle and of the first coordinate circle. It is straight-forward to deduce that the following matrices
\[
\begin{pmatrix}
1\\
0
\end{pmatrix},\
\begin{pmatrix}
-1\\
1
\end{pmatrix},\
\begin{pmatrix}
1 & 0\\
0 & 2
\end{pmatrix},
\]
give the homomorphisms $p_1^*, p_2^*$ and $p^*$ of lattices in the basis $f_1,f_2$, which are induced by the torus homomorphisms $p_1,p_2$ and by the natural projection $p\colon T^2\to T^2/(\Z/2\Z)$, respectively. There is the commutative diagram of torus homomorphisms and the induced commutative diagram of lattice embeddings
\[
\begin{tikzcd}
T^2 \arrow{r}{p_i}\arrow{d}{p} & S^1 \arrow{d}\\
T^2/(\Z/2\Z) \arrow{r}{q_i}  & S^1/(\Z/2\Z),
\end{tikzcd}
\begin{tikzcd}
\Z \arrow[hook]{r}{p_i^*}\arrow[hook]{d}{\cdot 2} & \Z^2 \arrow[hook]{d}{p^*}\\
\Z \arrow[hook]{r}{q_i^*}  & \Z^2,
\end{tikzcd}
\]
respectively. From this one deduces that the matrices
\[
\begin{pmatrix}
0\\
2
\end{pmatrix},\
\begin{pmatrix}
1\\
-2
\end{pmatrix},
\]
give the homomorphisms $q_1^*, q_2^*$ in the basis $f_1,f_2$.
\end{example}

\begin{construction}
Let $H_d=G_d^{I_0}$ for some $I_0\subseteq [m]$. Let $K$ be any simplicial complex on $[m]$. The $(\cat K)$-diagram $D_1:=Q(K,H_d)$ consists of objects $G_d^{[m]\setminus I_0}/G_d^{I\setminus I_0}$. Let $D_2$ be the $(\cat K/I_0)$-diagram $Q(K/I_0,1)$, where
\[
K/I_0:=\lb I\setminus I_0|\ I\in K\rb,
\]
denotes the \textit{excision of the simplicial complex} $K$ along $I_0$ \cite{bu-pa-15}. By the definition, one has $D_1=\alpha^* D_2$, where
\[
\alpha\colon \cat K\to \cat (K/I_0),\ I\mapsto I\setminus I_0,
\]
is the poset morphism and $\alpha^* D_2$ is the pullback of the diagram $D_2$ along $\alpha$. In general, the homotopy colimits of $D_1$ and of $D_2$ are not homotopy equivalent, as the example of the cone $K=\cone_{m} \tilde{K}$ with apex at $m\in [m]$ over $\tilde{K}$ on $[m-1]$, and $H_d=G_1^{\lb m\rb}$ shows. Indeed, the join $(D^d,S^{d-1})^{\tilde{K}} * pt\simeq \hoc D_1$ is contractible, whereas $\hoc D_2\simeq (D^d,S^{d-1})^{\tilde{K}}$ is not contractible, in general.
\end{construction}

The following lemma is straight-forward to prove.

\begin{lemma}\label{lm:coleq}
Let $D$ be a $(\cat K)$-diagram of closed subspaces in a topological space $X$ and let $G(I)$ be a $(\cat K)$-diagram of closed subgroups in a group $G$. Then one has the following equality
\[
\col (D(I)\times G/G(I))=((\col D(I))\times G)/\sim,
\]
\[
(x,g)\sim (x',g') \Leftrightarrow x=x'\in D(I)\ \&\ g^{-1}g'\in G(I).
\]
\end{lemma}

The following corollary was previously known in the case of partial quotients (see \cite[\S 2]{fr-10}, \cite[\S 4]{fr-21}). (For another variant of the corresponding generalization see \cite{ba-be-co-hi-17}).

\begin{corollary}\label{cor:djcon}
The quotient $(D^d,S^{d-1})^K/H_d$ is homeomorphic to the quotient
\[
\bigl((\col (D^d)^I/(H_d\cap G_d^I))\times (G_d^m/H_d)\bigr)/\sim,
\]
\[
(x,g)\sim (x',g') \Leftrightarrow x=x'\in (D^d)^I/(H_d\cap G_d^I)\ \&\ g^{-1}g'\in G_d^I/(H_d\cap G_d^I).
\]
\end{corollary}
\begin{proof}
The claim follows directly from Corollary \ref{cor:dirhomeo} and Lemma \ref{lm:coleq}.
\end{proof}

Let $K=\partial P^*$ be the dual simplicial sphere to a simple polytope $P^n\subset \R^n$. Suppose that the action of $H_d$ on $(D^d,S^{d-1})^K$ is free and that $H_d\cong G_d^{m-n}$ holds. The quotient $(D^d,S^{d-1})^K/H_d$ is called a \textit{small cover} for $d=1$ and a \textit{quasitoric manifold} for $d=2$, respectively (\cite{da-ja-91}). Let $G_d^n=G^m_d/H_d$ be the real or complex torus for $d=1$ or $d=2$, respectively. Consider the $(\cat K)$-diagram $G_d^n/p(G_d^I)$ (see Proposition \ref{pr:quotdiag} below for the precise definition), where $p\colon G_d^m\to G_d^n$ is the natural quotient homomorphism.

The following theorem was first proved for toric varieties in \cite{we-zi-zi-99} and then generalized to quasitoric manifolds in \cite{pa-ra-08}. We deduce it from Theorem \ref{thm:quothocolim} below.

\begin{theorem}\label{thm:qtorichoc}(\cite{we-zi-zi-99}, \cite{pa-ra-08})
Let $(K,H_d)$ be as above. Then there is the homotopy equivalence of spaces
\[
(D^d,S^{d-1})^K/H_d\simeq\hoc\ G^n_d/p(G_d^{I}).
\]
\end{theorem}
\begin{proof}
By the condition on $H_d$, one has $G_d^I\cap H_d =1$ for any $I\in K$. Then the commutative diagram \eqref{eq:maincommdiag} takes form
\begin{equation}\label{eq:anotherexseq}
\begin{tikzcd}
1 \arrow{r} & G_d^I \arrow{r}{p|_{G_d^I}}\arrow{d} & G_d^m/H_d \arrow{r}\arrow[equal]{d} & G_d^{m}/( G_d^I\cdot H_d) \arrow{r}\arrow{d} & 1\\
1 \arrow{r} & G_d^J \arrow{r}{p|_{G_d^J}} & G_d^m/H_d \arrow{r} & G_d^{m}/( G_d^J\cdot H_d) \arrow{r} & 1,
\end{tikzcd}
\end{equation}
for $I\subseteq J\in \cat K$. Hence, \eqref{eq:anotherexseq} defines the isomorphism $Q_d\to G^n_d/p(G_d^{I})$ of $(\cat K)$-diagrams. Now the desired homotopy equivalence follows by Theorem \ref{thm:quothocolim}.
\end{proof}

\begin{remark}\label{rem:dj}
In the case of a small cover $d=1$ the homeomorphism from Corollary \ref{cor:djcon} is similar to the Davis-Januszkiewicz construction \cite{da-ja-91}, because $(D^1,pt)^K$ coincides
with the cubical subdivision of the simple polytope $P\subset\R^n$, where $K=\partial P^{*}$. By relaxing the condition on $H_d\cong G_d^{m-n}$ to the condition of only finite stabilizers of the respective action on $(D^d,S^{d-1})^K$ one obtains the homeomorphism similar to the one that was proved for quasitoric orbifolds in \cite{po-sa-10}. For general partial quotients a similar homeomorphism was proved in \cite{fr-10}, \cite{fr-21} (also see \cite{ba-be-co-hi-17}). Notice that the homeomorphism from Corollary \ref{cor:djcon} is in general different from the Davis-Januszkiewicz construction \cite{da-ja-91}. For a partial quotient the natural projection of orbit spaces
\[
(D^d,S^{d-1})^I/H_d\to (D^d,S^{d-1})^I/G_d^m=(I^1,1)^I,\ I\in \cat K,
\]
is a trivial $(G_{d}^m/(H_d\cdot G_d^I))$-bundle, where $I^1=[0,1]\subset \R$. This observation allows to recover the genuine homeomorphism from Davis-Januszkiewicz construction.
\end{remark}

\section{Equivariant homotopy colimits and G-CW-complexes}\label{sec:gcw}
In this section we prove the strengthening of Theorem \ref{thm:quothocolim} in the equivariant setting leading to G-CW-approximation for quotients of moment-angle complexes.

Let $G$ be a topological group.

\begin{definition}\cite{wa-80}
The equivariant union $X=\underset{n\in\Z_{\geq 0}}{\col} X_n$ of $G$-spaces $X_n$ is called a \textit{$G$-complex} if there is a pushout
\[
X_{n+1}=X_n\bigcup_{\varphi_n} \biggl(\bigsqcup_{\alpha\in A_n} D^{n_\alpha}\times G/H_{\alpha}\biggr),
\]
of $G$-spaces with the natural left $G$-action (left $G$-action on $G/H_{\alpha}$ and trivial action on $D^{n_{\alpha}}$), where
\[
\varphi_n\colon \bigsqcup_{\alpha\in A_n} S^{n_\alpha-1}\times G/H_{\alpha}\to X_n,
\]
is $G$-equivariant, $D^{n_\alpha}$ is an $n_\alpha$-dimensional disk and $\lb H_{\alpha}\rb_{\alpha\in A_n}$ is a collection of closed subgroups in $G$. If $n_\alpha=n$ holds for any $\alpha\in A_n$, then $X$ is called a \textit{$G$-CW-complex}.
\end{definition}

The category $GTop$ of $G$-spaces (objects) and $G$-equivariant maps between these spaces (morphisms) has the Quillen model structure given by $G$-equivariant weak equivalences, $G$-equivariant Serre fibrations and $G$-equivariant retracts of $G$-CW-complexes \cite{bo-ka-72}. Since $\cat K$ is a Reedy category \cite{dw-sp-95}, the category $GTop^{\cat K}$ has the Reedy model structure given by objectwise weak equivalences and fibrations, and cofibrations are given by morphisms $D\to E$ such that $D(I)\sqcup_{L_{D}(I)}L_{E}(I)\to E(I)$ is a cofibration for any $I\in\cat K$, where
\[
L_{D}(I):=\underset{(\cat K)_{<I}}{\hoc} D\to D(I),
\]
is the natural map for $D$ \cite{dw-sp-95}, \cite{bu-pa-15}.

The following proposition is straight-forward to prove.

\begin{proposition}\label{pr:quotdiag}
Let $G_{I}$ be a collection of closed subgroups in $G$ such that $G_{I}\subseteq G_{J}$ holds for any $I\subseteq J\in \cat K$. Define the $(\cat K)$-diagram $D$
\[
D(I):=G/G_{I},\ D(I\to J)\colon G/G_{I}\to G/G_{J},
\]
where $D(I\to J)$ is the natural projection. Then $D$ is fibrant in $GTop^{\cat K}$ and its homotopy colimit in $GTop$ is given by
\begin{equation}\label{eq:hoc}
\hoc\ D=\bigl(\bigsqcup_{I\in \cat K} D(I)\times |K_{\geq I}|\bigr)/\sim,
\end{equation}
where by definition $(d,\Inc_{I\to I'}(I'))\sim (D(I,I')(d),I')$ and $\Inc_{I\to I'}\colon |K_{\geq I'}|\to |K_{\geq I}|$ is the natural embedding. Furthermore, the decomposition \eqref{eq:hoc} endows $\hoc\ D$ with the structure of a $G$-CW-complex.
\end{proposition}

Let $G=G_d^m/H_d$. The natural $G$-action on $(D^d,S^{d-1})^{K}/H_d$ allows to consider the $(\cat K)$-diagram $(D^d,S^{d-1})^{I}/H_d$ as a $(\cat K)$-diagram in $GTop$.

\begin{theorem}\label{thm:equivquothocolim}
For any closed subgroup $H_d$ in $G_d^m$ and any simplicial complex $K$ on $[m]$ there is the $G_d^m/H_d$-equivariant homotopy equivalence
\[
(D^d,S^{d-1})^{K}/H_d\simeq \hoc\ G_d^m/( G_d^I\cdot H_d).
\]
\end{theorem}
\begin{proof}
Follows from Theorem \ref{thm:quothocolim} by the standard properties of a homotopy colimit by using $G$-equivariance of all arrows in \eqref{eq:fibemb}.
\end{proof}

\begin{remark}
Theorem \ref{thm:equivquothocolim} gives an explicit $G$-CW-approximation of the quotient $(D^d,S^{d-1})^{K}/H_d$ with cells
\[
\Delta^{s}(I_0\supset\cdots\supset I_s)\times G_d^m/( H_d\cdot G_d^I),\ I_0\in \cat K,\ s\geq 0.
\]
In the case of partial quotients this decomposition was previously known, see \cite[\S 2]{fr-10}. The homeomorphism from Corollary \ref{cor:djcon} is easily shown to be $G$-equivariant.
\end{remark}

\section{Equivariant cohomology of quotients for moment-angle complexes}\label{sec:borel}

In this section we study the formality problem for the Borel construction of the natural $L_d$-action on the quotient $(D^d,S^{d-1})^K/H_d$ of the moment-angle complex for any closed subgroup $H_d$ satisfying the Condition \ref{cond:1}. This condition is satisfied for any freely acting subgroup $H_d$ on the corresponding moment-angle complex. Notice that the formality of the corresponding Borel space for any freely acting subgroup (that is, for any partial quotient) was proved in \cite{no-ra-05}. We follow the ideas of \cite{no-ra-05} in our proof of formality for a wider class of actions.


\subsection{On the Borel construction}\label{ssec:borelc}

The classifying space functor $B\colon TGrp\to Top$ (for example, see \cite{pa-ra-08}) induces the functor $TGrp^{\cat K}\to Top^{\cat K}$ which we denote by $B$ by a slight abuse of the notation. Thus, there are the well-defined $(\cat K)$-diagrams $BS_d$, $\kappa(BL_d)$ and $BQ_d$. Consider the $(\cat K)$-diagram $EL_d\times_{L_d} Q_d$ given by applying the (functorial) Borel construction to the diagram $Q_d$ of $L_d$-spaces.

\begin{proposition}\label{pr:equivheq}
There is the $L_d$-equivariant homotopy equivalence of $(\cat K)$-diagrams
\[
EL_d\times_{L_d} Q_d\to BS_d.
\]
\end{proposition}
\begin{proof}
Follows directly from the fact that $L_d$ acts on the space $Q_d(I)=G_d^{m}/( H_d\cdot G_d^I)$ transitively with the kernel $S_d=G_d^I/(G_d^I\cap H_d)$ by \eqref{eq:maincommdiag}.
\end{proof}

Thus, the Borel construction for the natural $L_d$-action on $(D^d,S^{d-1})^K/H_d$ takes the following form.

\begin{corollary}\label{cor:borelconstr}
For $G=L_d$, there is the following fibration in $GTop^{\cat K}$:
\begin{equation}\label{eq:sfd}
\begin{tikzcd}
Q_d \arrow{r} & BS_d \arrow{r} & \kappa(BL_d)
\end{tikzcd}.
\end{equation}
\end{corollary}

\begin{theorem}\label{thm:eqcoh}
There is the following homotopy equivalence
\[
EL_d\times_{L_d} (D^d,S^{d-1})^K/H_d\simeq \col B(G_d^I/(G_d^I\cap H_d)),
\]
for the Borel construction of the $L_d$-action on the quotient $(D^d,S^{d-1})^K/H_d$ of the moment-angle complex.
\end{theorem}
\begin{proof}
Follows directly from Proposition \ref{pr:equivheq} by the Homotopy Lemma of \cite{we-zi-zi-99}.
\end{proof}

\begin{example}\label{ex:sr}
Suppose that the $H_d$-action on $(D^d,S^{d-1})^K$ is free. Then it follows from the standard properties of equivariant cohomology that the cohomology ring isomorphism (with $\Z$-coefficients)
\[
H^{*}_{L_d}((D^d,S^{d-1})^K/H_d)\cong H^{*}_{G_d^m}((D^d,S^{d-1})^K),
\]
takes place. On the other hand, freeness of the action implies that $H_d\cap G_d^I$ is a trivial group for any $I\in \cat K$. Hence, the colimit of $BS_d$ is the Davis-Januszkiewicz space $DJ(K)=(\F_d P^{\infty},pt)^K$ whose cohomology ring with $R_d$-coefficients is isomorphic to the Stanley-Reisner ring $R_d[K]$ \cite{da-ja-91}. Thus Theorem \ref{thm:eqcoh} gives a correct answer in this case by \cite{bi-dc-pr-90} (for $d=1$ we take reduction of integral cohomology coefficients modulo two).
\end{example}

As an application of Corollary \ref{cor:borelconstr} and Theorem \ref{thm:eqcoh} we describe the Borel construction for the quotient by any coordinate subgroup in $G_d^m$ (with not necessarily free action) below.

\begin{corollary}\label{cor:coordsubq}
Let $H_d=G_d^{I_0}$ for an arbitrary fixed $I_0\subseteq [m]$. Then for any complex $K$ the Borel construction of $L_d$-action on $(D^d,S^{d-1})^K/H_d$ is homotopy equivalent to the real or complex Davis-Januszkiewicz space, $\R DJ(K/I_0)$ or $DJ(K/I_0)$, for $d=1$ or $d=2$, respectively. Furthermore, one has the ring isomorphism:
\[
H^{*}_{L_d}((D^d,S^{d-1})^K/H_d;\ R_d)\cong R_d[K/I_0].
\]
\end{corollary}
\begin{proof}
Notice that the natural group isomorphism
\begin{equation}\label{eq:natiso}
G_d^I/(G_d^I\cap G_d^{I_0})\cong G_d^{I\setminus I_0},
\end{equation}
holds for any $I\in\cat K$. Hence, the following diagram
\[
\begin{tikzcd}
G_d^I/(G_d^I\cap G_d^{I_0}) \arrow{r} \arrow[swap]{d}{S(I\to J)} & G_d^{I\setminus I_0} \arrow{d}\\
G_d^I/(G_d^J\cap G_d^{I_0}) \arrow{r} & G_d^{J\setminus I_0},
\end{tikzcd}
\]
where both horizontal arrows are given by \eqref{eq:natiso} and the right vertical arrow is the standard embedding, is commutative for any $I\subseteq J\in\cat K$. This diagram yields the isomorphism of $(\cat K)$-diagrams $BS_d$ and $BG_d^{I\setminus I_0}$. Hence, one has
\begin{equation}\label{eq:bsdequal}
\col BS_d\cong
\underset{\cat K}{\col\ }BG_d^{I\setminus I_0}=
\underset{\cat K/I_0}{\col\ }BG_d^{I\setminus I_0}.
\end{equation}
The last equality holds because the $(\cat K)$-diagram $BG_d^{I\setminus I_0}$ is cofibrant and has a singleton (a point) as the object corresponding to any $I\subset I_0$ such that $I\in K$ holds. The last expression in \eqref{eq:bsdequal} is the real or complex Davis-Januszkiewicz space by the definition for $d=1$ or $d=2$, respectively. This proves the first claim. The second claim then follows from the first by the standard computation for moment-angle complexes, see Example \ref{ex:sr}. The proof is complete.
\end{proof}

\begin{example}
The Borel construction of $L_d$-action on $(D^d,S^{d-1})^K/H_d$ for $H_d=\lb 1\rb$ and $H_{d}=G_d^m$ is up to homotopy equivalence is the respective Davis-Januszkiewicz space $(\F_d P^{\infty},\pt)^K$ and the point, in according to Corollary \ref{cor:coordsubq}, respectively.
\end{example}

\subsection{On a certain class of quotients for moment-angle complexes}\label{ssec:certaincl}

Let $K$ be any simplicial complex on $[m]$. Let $H_d$ be any closed subgroup in $G_d^m$. We introduce the following condition on the pair $(K,H_d)$.

\begin{condition}\label{cond:1}
For any $I\subseteq J\in\cat K$ the subgroup $H_d\cap G_d^J$ maps to the subgroup $H_d\cap G_d^I$ under the natural projection $G_d^J\to G_d^I$. Or equivalently, in the following diagram there exists an upper horizontal arrow making it a commutative diagram
\begin{equation}\label{eq:condact}
\begin{tikzcd}
G_d^J\cap H_d \arrow{r}\arrow[hook]{d} & G_d^I \cap H_d \arrow[hook]{d}\\
G_d^J \arrow{r}  & G_d^I,
\end{tikzcd}
\end{equation}
where the lower horizontal arrow is the natural projection.
\end{condition}

\begin{example}
For any $K$ and any $H_d$ such that $H_d$ acts freely on $(D^d,S^{d-1})^K$ both groups in the upper row of \eqref{eq:condact} are trivial. Hence, Condition \ref{cond:1} holds for any free action of $H_d$ on $(D^d,S^{d-1})^K$.
\end{example}

\begin{example}
Let $H_d=G_d^{I_0}$ for any fixed $I_0\subseteq [m]$ and let $K$ be any simplicial complex. Notice that $G_d^I\cap H_d=G_d^{I\cap I_0}$ holds for any $I\subseteq [m]$. The natural projection $G_d^J\to G_d^I$ sends $i$-th coordinate subgroup $(G_d)_i$ to $1$ if $i\not\in I$ and acts as an identity if $i\in I$. Hence, the image of $G_d^J\cap H_d$ under this projection coincides with $G_d^I\cap H_d$. We conclude that Condition \ref{cond:1} holds for the action of a coordinate subgroup $H_d=G_d^{I_0}$ on $(D^d,S^{d-1})^K$. Notice that this action is not free, in general.
\end{example}

\begin{example}
In \cite{li-19} a certain class of closed subgroups $H_d$ in $G_d^m$ acting on $(D^d,S^{d-1})^K$ was introduced. One can check that for $d=2$, $m=2$ and $K=\Delta^1$ the natural action of the diagonal circle $H_d=S^1_d$ on $(D^d,S^{d-1})^K$ belongs to this class and does not satisfy Condition \ref{cond:1} for $I=\lb 1\rb$, $J=\lb 1,2\rb$.
\end{example}

\subsection{Twin diagrams and equivariant cohomology}\label{ssec:cont}

Let $D$ and $D^{\vee}$ be $(\cat K)$- and $(\cat^{op} K)$-diagrams with values in the category $\Top$, respectively. Suppose that $D(I)=D^{\vee}(I)$ holds for any $I\in \cat K$. Recall the following definition.

\begin{definition}\cite{no-ra-05}\label{defn:twindef}
The diagrams $D$, $D^{\vee}$ are called \textit{twin diagrams} if the identity
\[
D^{\vee}(J\to I')\circ D(I\to J)=D(I\cap I'\to I')\circ D^{\vee}(I\to I\cap I'),
\]
holds for any $I,I'\subseteq J$ in $\cat K$.
\end{definition}

Recall that any $(\cat K)$-diagram $D$ of pointed topological spaces gives rise to the Bousfield-Kan type cohomological (with integral coefficients) spectral sequence $(E_{D})_{r}^{s,t}$ (see \cite{no-ra-05})
\[
(E_{D})_{2}^{i,j}={\lim}^{i}\tilde{H}^j(D)\Rightarrow \tilde{H}^{i+j}(\hoc\ D).
\]

\begin{theorem}\cite[p.39, Lemma 3.8, p.41, Theorem 3.10]{no-ra-05}
Suppose that a $(\cat K)$-diagram $D$ is cofibrant and has a twin. Then the second page of the Bousfield-Kan spectral sequence $(E_{D})_{2}^{s,t}$ of $D$ is concentrated at $s=0$. In particular, $(E_{D})_{r}^{s,t}$ collapses at the second page $r=2$.
\end{theorem}

\begin{corollary}\cite[p.42, Corollary 3.12]{no-ra-05}\label{cor:convcoftwin}
If $D$ is cofibrant and has a twin, then one has
\[
\tilde{H}^{i}(\col\  D)=\lim\  \tilde{H}^{i}(D),\
{\lim}^j\ \tilde{H}^{i}(D)=0,\ j>0.
\]
\end{corollary}

\begin{theorem}\label{thm:limcoh}
Suppose that $K$ and $H_d$ satisfy Condition \ref{cond:1}. Then one has
\[
\tilde{H}^{i}(\col\  D)=\lim\  \tilde{H}^{i}(D),\
{\lim}^j\  \tilde{H}^{i}(D)=0,\ j>0,
\]
where $D=S_d, BS_d$. In particular, $\tilde{H}^{odd}(\col\ BS_2;\ \Z)=0$.
\end{theorem}
\begin{proof}
By the condition, the subgroup $H_{d}\cap G_{d}^J$ maps to the subgroup $H_{d}\cap G_{d}^I$ under the natural projection $G_{d}^J\to G_{d}^I$. Hence, there is a well-defined $(\cat^{op} K)$-diagram $S_{d}^{\vee}$, where $S_{d}^{\vee}(J\to I)\colon G_{d}^{J}/(H_{d}\cap G_{d}^J)\to G_{d}^{I}/(H_d\cap G_{d}^I)$ is induced by the natural projection $G_{d}^J\to G_{d}^I$. Define $(BS_{d})^{\vee}:=B (S_{d}^{\vee})$.

We check that the pairs $(S_{d},S_{d}^{\vee})$ and $(BS_{d},(BS_{d})^{\vee})$ are pairs of twin diagrams. Let $I,I'\subseteq J\in K$ be arbitrary. Consider the following diagram:
\begin{equation}\label{eq:twincommdiag}
\begin{tikzcd}[sep=.3cm]
& 1\arrow{rr} && H_d\cap G_d^{I\cap I'}  \arrow{rr}\arrow{dd} && G_d^{I\cap I'} \arrow{dd}\arrow{rr} && S_d(I\cap I')\arrow{dd}[yshift=10.0]{S}\arrow{rr} && 1\\
1\arrow{rr} && H_d\cap G_d^I\arrow[crossing over]{rr}\arrow{dd}\arrow{ru} && G_d^I\arrow{dd}\arrow{ru}\arrow[crossing over]{rr} && S_d(I)\arrow{dd}[yshift=10.0]{S}\arrow{ru}{S^{\vee}}\arrow[crossing over]{rr}&& 1 &\\
& 1\arrow{rr} &&  H_d\cap G_d^{I'}\arrow{rr} &&  G_d^{I'}\arrow{rr} && S_d(I') \arrow{rr} && 1.\\
1\arrow{rr} && H_d\cap G_d^J \arrow[crossing over, from=uu]\arrow{rr}\arrow{ru} && G_d^J\arrow{rr}\arrow[crossing over, from=uu]\arrow{ru}\arrow{rr} && S_d(J)\arrow[crossing over, from=uu]\arrow{ru}{S^{\vee}}\arrow{rr} && 1 &
\end{tikzcd}
\end{equation}
In this diagram the horizontal short sequences are exact and are given by the embedding of the corresponding subgroup and by the quotient by its image. The remaining horizontal and vertical arrows are given by the corresponding projections and embeddings, respectively. Consider the left cube of this diagram. The front and the back faces of it are commutative by the definition. Next, the upper and the lower faces of the cube are commutative by Condition \ref{cond:1}. A straight-forward check shows that the right face of this cube is commutative. This implies that the left face of the left cube in \eqref{eq:twincommdiag} is commutative by appealing to monomorphic property of arrows from the horizontal short exact sequences in \eqref{eq:twincommdiag}. Therefore the left cube in \eqref{eq:twincommdiag} is commutative. By Lemma \ref{lm:stupid} this implies that the right cube in \eqref{eq:twincommdiag} is well-defined and commutative. Hence, the right square of the diagram \eqref{eq:twincommdiag} is commutative. It remains to notice that this square coincides with the square from the condition of the Definition \ref{defn:twindef}, see \eqref{eq:bigdiag} (any arrow of this square is labelled by the corresponding diagram).  It follows that the pair $(S_{d},S_{d}^{\vee})$ is a pair of twin diagrams. The proof of the claim that the pair $(BS_{d},(BS_{d})^{\vee})$ is a pair of twin diagrams is obtained from this fact by applying the functor of the classifying space $B$ to the diagram \eqref{eq:twincommdiag}. Clearly, the diagrams $S_{d}$, $BS_{d}$ are cofibrant. We conclude that the necessary claim follows from Corollary \ref{cor:convcoftwin}.
\end{proof}

\begin{corollary}\label{cor:eqcohdec}
Suppose that $K$ and $H_d$ satisfy Condition \ref{cond:1}. Then one has the following ring isomorphism:
\[
\tilde{H}^{*}_{L_d}((D^d,S^{d-1})^K/H_d)\cong \lim\  \tilde{H}^{*} (B (G_d^I / (G_d^I\cap H_d))).
\]
In particular, one has $\tilde{H}^{odd}_{L_2}(\mathcal{Z}_K /H_2;\ \Z)=0$.
\end{corollary}

\begin{example}\label{ex:freelim}
Let $H_d$ be a freely acting closed subgroup in $G_d^m$ on $(D^d,S^{d-1})^K$ (for example, $H=\lb 1\rb$). In this case the corresponding Borel construction is homotopy equivalent to the Davis-Januszkiewicz space $(\F_d P^{\infty},\pt)^K$. The cohomology ring of the latter space is isomorphic to the Stanley-Reisner ring (see \cite{bi-dc-pr-90, da-ja-91}). Hence, there is the ring isomorphism \cite{bu-pa-15}
\[
\tilde{H}^{*}_{L_d}((D^d,S^{d-1})^K/H_d)\cong R_d[K].
\]
It is well known that
\[
R_d[K]\cong \underset{I=(i_1,\dots,i_q)\in \cat K}{\lim}\  R_d[v_{i_1},\dots,v_{i_q}],
\]
holds for the Stanley-Reisner ring \cite{bu-pa-15}, where the arrows are the obvious monomorphisms to the polynomial ring $R_d[v_1,\dots,v_m]$, where $\deg v_j:=d_i$. Therefore, the group
\[
\lim\ \tilde{H}^{i} (BG_{d}^{I};\ R_d)\cong \underset{I=(i_1,\dots,i_q)\in \cat K}{\lim}\  (R_d[v_{i_1},\dots,v_{i_q}])_{i},
\]
agrees with the respective component $(R_d[K])_{i}$ of the Stanley-Reisner ring (for $d=1$ take reduction of coefficients modulo $2$).
\end{example}

\subsection{Formality of the Borel construction for the class of quotients for moment-angle complexes and Eilenberg-Moore spectral sequence}\label{ssec:formality}

In this section we study only quotients of complex moment-angle complexes (that is, $d=2$) and consider only cohomology with integral coefficients due to usage of the Eilenberg-Moore spectral sequences. For brevity we omit the subscript $d$ and replace $G_2$ with $T=S^1$ everywhere below. It is crucial that everywhere in \S\ref{ssec:formality} we assume that Condition \ref{cond:1} holds for the pair $(K, H)$.

The proof of the following proposition is similar to the proofs of \cite[p.44, Lemma 4.7, p.42]{no-ra-05} (notice that the analogue of \cite[p.42, Corollary 3.12]{no-ra-05} is given in \S\ref{ssec:cont}).

\begin{proposition}\cite{no-ra-05}\label{pr:qiso1}
Suppose that Condition \ref{cond:1} holds for the pair $(K, H)$. Then the natural homomorphism $g\colon C^*(\col\  BS)\to \lim\ C^*(BS)$ is a quasiisomorphism in $\dga_{\Z}$, and the edge homomorphism (see \cite{no-ra-05}) $h\colon H^*(\col\  BS)\to \lim\ H^*(BS)$ is an isomorphism in $\dga_{\Z}$, where $C^*(\col\  BS)$ is the normalized singular cochain complex of $\col\  BS$.
\end{proposition}

In the following we need to describe the induced morphisms of chain and cochain complexes for tori under corresponding torus homomorphisms. Recall that for any complex compact torus $T=(S^1)^r$ there are the simplicial sets $B(BN)$, $BN$ given by
\[
B(BN)_n:=\lb [b_{n-1},\dots,b_0]|\ b_i\in BN_i\rb,\ BN_n:=\lb [a_0,\dots,a_{n-1}]|\ a_0,\dots,a_{n-1}\in N\rb,\ n>0,
\]
and $B(BN)_0=BN_0:=\lb [\ ]\rb$, where $N=\pi_1(T,e)\simeq \Z^r$ and the explicit formulas for faces and degenerations of $B(BN)$ are given in \cite[p.87]{ma-92}. Any homomorphism of tori $f\colon T\to T'$ of ranks $r$, $r'$, respectively, is given by the formula
\[
(e^{2\pi i \varphi_{1}},\dots, e^{2\pi i \varphi_{r}})\mapsto
(e^{2\pi i \sum_{q} a_{1,q}\varphi_{q}},\dots, e^{2\pi i \sum_{q} a_{r',q}\varphi_{q}}),\ \varphi_{1},\dots,\varphi_{r}\in[0,1),
\]
for some integer matrix $A=(a_{i,j})\in \Mat_{r',r}(\Z)$, which we denote by $A=A(f)$, and vice versa. The bar-construction of the simplicial set $B(BN)$ gives the complexes whose homology and cohomology compute the integral homology and cohomology of $BT$, respectively \cite{ma-92}. We denote the corresponding simplicial set, and chain and cochain complexes by $\overline{W}(BT)=B(BN)$, $\overline{W}_{*}(BT)$, $\overline{W}^*(BT)$, respectively. (Notice that in \cite{ma-92} $\overline{W}(BT)$ is denoted by $\overline{W}(T)$.) Explicitly, the induced by $f$ morphism $f_{*}\colon \overline{W}(BT)\to \overline{W}(BT')$ is given by
\[
f_{*}([b_{n-1},\dots,b_0])=[f_{*}(b_{n-1}),\dots,f_{*}(b_0)],\ f_{*}([a_0,\dots,a_{n-1}])=[f_{*}(a_0),\dots,f_{*}(a_{n-1})],
\]
where $f_{*}(a)=A a$ is the lattice morphism given by the matrix $A$.

\begin{lemma}\label{lm:qisosimp}
There is the following commutative diagram of simplicial sets
\[
\begin{tikzcd}
\overline{W}(BT)\arrow{r}\arrow{d}{f_*} & S(BT) \arrow{d}{f_*}\\
\overline{W}(BT')\arrow{r} & S(BT'),
\end{tikzcd}
\]
where any horizontal arrow is a homotopy equivalence and $S^*(BT)$ is the simplicial set of singular simplices in $BT$.
\end{lemma}
\begin{proof}
Notice that the simplicial set $\overline{W}(BT)$ coincides with the bar construction $B(*,BN,*)$ of the simplicial group $BN$ \cite[\S 21]{ma-92}. By \cite[Lemma 3.3]{fr-21}, there exists a homotopy equivalence of simplicial sets
\[
F\colon BN\to S(T),
\]
which is natural with respect to morphisms of the torus $T$. The map $F$ induces the homotopy equivalence
\begin{equation}\label{eq:1heq}
F_*\colon \overline{W}(BT)=B(*,BN,*)\to B(*,S(T),*),
\end{equation}
by naturality of the bar construction, which is again functorial with respect to morphisms of the group $T$. By \cite[\S 13]{ma-75}, there is a homotopy equivalence
\begin{equation}\label{eq:2heq}
B(*,S(T),*) \to S(BT),
\end{equation}
which is also natural with respect to morphisms of the torus $T$. One obtains the desired commutative diagram by taking the composition of \eqref{eq:1heq}, \eqref{eq:2heq} and using the aforementioned naturality properties.
\end{proof}

In the following we consider a simplicial variant of the argument from \cite{no-ra-05}. Choose a generator $v$ of the ring $H^* (BS^1)\cong \Z[v]$ for the Eilenberg-Maclane space $BS^1=\C P^{\infty}$. Choose a cocycle $\psi_{v}\in \overline{W}^2(BS^1)$ representing $v$ in sense of homotopy equivalence from Lemma \ref{lm:qisosimp}. Let $\psi\colon H^* (BS^1) \to \overline{W}^*(BS^1)$ be given by $v^q\mapsto (\psi_{v})^q$, where the product on the right is with respect to $\cup_{1}$-product. For a complex compact torus $T=(S^1)^r$ let
\[
\begin{tikzcd}
\kappa\colon H^* (BT)\arrow{r}{\cong} & H^* (BS^1)^{\otimes r} \arrow{r}{\otimes \psi} & \overline{W}^* (BS^1)^{\otimes r},
\end{tikzcd}
\]
be given by the composition of K\"unneth isomorphism and $\otimes \psi$. There is the following zig-zag of quasiisomorphisms
\[
\begin{tikzcd}
\Hm (\overline{W}_{*} (BS^1);\ \Z)^{\otimes r}\arrow{r} & \Hm (\overline{W}_* (BS^1)^{\otimes r};\ \Z) & \overline{W}^* (BT) \arrow[swap]{l}{ez^*},
\end{tikzcd}
\]
where $ez^*$ is the dual to the Eilenberg-Zilber map. For a group homomorphism $f\colon T\to T'$ with the corresponding matrix $A=(a_{i,j})\in \Mat_{r',r}(\Z)$, let
\[
A_*\colon \overline{W}_{*} (BS^1)^{\otimes r}\to \overline{W}_{*} (BS^1)^{\otimes r'},\
A^*\colon \overline{W}^{*} (BS^1)^{\otimes r'}\to \overline{W}^{*} (BS^1)^{\otimes r},
\]
be induced on degree $2$ by the lattice morphisms with matrices $A$ and $A^{t}$, respectively.

\begin{lemma}\label{lm:functqiso}
Let $f\colon T\to T'$ be a homomorphism of compact complex tori of dimension $r$, $r'$, respectively. Then there is the following commutative diagram
\begin{equation}\label{eq:nora}
\begin{tikzcd}
H^* (BT') \arrow{r}{\kappa} \arrow{d}{f^*} & \overline{W}^*(BS^1)^{\otimes r'} \arrow{r} \arrow{d}{A^*} & \Hm(\overline{W}_{*} (BS^1)^{\otimes r'};\ \Z) \arrow{d}{\Hm(A_*,\Id)} & \overline{W}^* (BT') \arrow[swap]{l}{ez^*} \arrow{d}{f^*} &\\
H^* (BT) \arrow{r}{\kappa} & \overline{W}^*(BS^1)^{\otimes r} \arrow{r} & \Hm(\overline{W}_{*} (BS^1)^{\otimes r};\ \Z) & \overline{W}^* (BT) \arrow[swap]{l}{ez^*} &\\
\end{tikzcd}
\end{equation}
where any horizontal arrow is a quasiisomophism.
\end{lemma}
\begin{proof}
Let $H^{*}(BT)\cong\Z[v_{1},\dots,v_{r}]$, $H^{*}(BT')\cong\Z[v'_{1},\dots,v'_{r}]$ be the isomorphisms given by K\"unneth isomorphism. The formulas
\begin{equation}\label{eq:invimfor}
A^*(\psi_{v'_i})=\sum_{q} a_{i,q} \psi_{v_q},\ f^*(v'_i)=\sum_{q} a_{i,q} v_q,
\end{equation}
give (by taking products) the two left vertical arrows in \eqref{eq:nora}. This proves commutativity of the left square in \eqref{eq:nora}. Choose the generator $u\in H_{2}(BS^1)$ and let $\varphi_{u}\in \overline{W}_{2}(BS^1)$ be a cycle representing $u$ in sense of homotopy equivalence from Lemma \ref{lm:qisosimp}. Then $\overline{W}^* (BS^1)\to\Hm(\overline{W}_* (BS^1);\ \Z)$ is given by $\psi_{v}\mapsto (\varphi_{u})^*$, where $(\varphi_{u})^*$ is the dual character to $\varphi_{u}$. Let $H_{*}(BT)\cong\Z[u_{1},\dots,u_{r}]$, $H^{*}(BT')\cong\Z[u'_{1},\dots,u'_{r}]$ be the isomorphisms given by K\"unneth isomorphism. The formulas
\begin{equation}\label{eq:invimforhom}
f_*(u_i)=\sum_{q} a_{q,i}u'_{q},\ \Hm(A_*,\Id)((\varphi_{u_i})^*)=\sum_{q} a_{q,i} (\varphi_{u'_q})^*,
\end{equation}
determine $f_*$ and $\Hm(f_*,\Id)$, respectively. Then the commutativity of the middle square in \eqref{eq:nora} follows directly from \eqref{eq:invimfor} and \eqref{eq:invimforhom}. Notice that there is the following diagram
\begin{equation}\label{eq:commez}
\begin{tikzcd}
\overline{W}_*(BS^1)^{\otimes r} \arrow{d}{A_*} \arrow{r}{ez}  & \overline{W}_* (BT) \arrow{d}{f_*}\\
\overline{W}_*(BS^1)^{\otimes r'} \arrow{r}{ez} & \overline{W}_* (BT'),
\end{tikzcd}
\end{equation}
where $ez$ is the Eilenberg-Zilber map. The explicit formula for the Eilenberg-Zilber map \cite[\S 29]{ma-92} and \eqref{eq:invimforhom} imply that the diagram \eqref{eq:commez} is commutative. This implies the commutativity of the right square in \eqref{eq:nora}. The proof is complete.
\end{proof}

\begin{corollary}\label{cor:zigzagdiag}
Suppose that Condition \ref{cond:1} holds for the pair $(K, H)$. Then there is the following zigzag
\begin{equation}\label{eq:zigzaglim}
\begin{tikzcd}
\lim H^*(BS) \arrow{r} & \lim D_1 \arrow{r} & \lim D_2 & \lim \overline{W}^*(BS) \arrow{l} & \lim C^*(BS)\arrow{l},
\end{tikzcd}
\end{equation}
of quasiisomorphisms, where
\[
D_1(I):=\overline{W}^*(BS^1)^{\otimes \rk S(I)},\ D_1(I\to J):=A(S(I\to J))^*,
\]
\[
D_2(I):=\Hm(\overline{W}_* (BS^1)^{\otimes \rk S(I)};\ \Z),\ D_2(I\to J):=\Hm(A(S(I\to J))_*,\Id).
\]
\end{corollary}
\begin{proof}
Notice that $S(I\to J)$ is a monomorphic homomorphism of tori and, therefore, $A(S(I\to J))$ is a monomorphic lattice homomorphism. This implies that any arrow in any diagram from in \eqref{eq:zigzaglim} is an epimorphism. Therefore, any such diagram is fibrant in the model category $\dga_{\Z}$, see \cite[Appendix C.1]{bu-pa-15}. Hence, any limit from \eqref{eq:zigzaglim} is quasiisomorphic to the respective homotopy limit, see \cite[Appendix C.1]{bu-pa-15}. It remains to use Lemma \ref{lm:functqiso}, because any quasiisomorphism of diagrams in $[\cat^{op} K,\dga_{\Z}]$ induces quasiisomorphism of the respective homotopy limits (see \cite[Appendix C.1]{bu-pa-15}).
\end{proof}

\begin{theorem}\label{thm:form}
Suppose that Condition \ref{cond:1} holds for the pair $(K, H)$. Then the differential graded algebra $C^*(\col\ BS)$ is formal in $\dga_{\Z}$.
\end{theorem}
\begin{proof}
Combining Proposition \ref{pr:qiso1} and Corollary \ref{cor:zigzagdiag} yields the zigzag
\[
\begin{tikzcd}[sep=.5cm]
H^* (\col BS) \arrow{r}{h} & \lim H^*(BS) \arrow{r} & \lim D_1 \arrow{r} & \lim D_2 & \lim \overline{W}^*(BS) \arrow{l} & \lim C^*(BS)\arrow{l} & C^* (\col BS) \arrow[swap]{l}{g}.
\end{tikzcd}
\]
The proof is complete.
\end{proof}

For a Serre fibration $p\colon E\to B$ with a connected fiber $F$ the Eilenberg-Moore spectral sequence $(E^{*,*}_{*},d)$ of the fiber inclusion has the second page \cite[p.233]{mcc-01}
\[
E^{n,s}_{2}=\Tor^{n,s}_{H^*(B)}(H^*(E),\Z),
\]
where the first grading is cohomological and the second is inner. If $B$ is simply-connected, then $(E^{*,*}_{*},d)$ converges strongly to $H^*(F)$, see \cite[p.233]{mcc-01}.

\begin{theorem}\label{thm:em}
Suppose that Condition \ref{cond:1} holds for the pair $(K, H)$. Then the Eilenberg-Moore spectral sequence for the fiber inclusion to the Borel construction of the $L$-action on $\mathcal{Z}_{K}/H$ is isomorphic to
\[
\Tor^{i,j}_{H^*(BL)}(\lim H^*(BS);\Z)\Rightarrow H^{i+j} (\mathcal{Z}_K /H).
\]
It collapses at the second page. In particular, the associated graded algebra of $H^{*}(\mathcal{Z}_{K}/H)$ is isomorphic to $\Tor^{*}_{H^*(BL)}(\lim H^*(BS);\Z)$.
\end{theorem}
\begin{proof}
The Eilenberg-Moore spectral sequence in question has the second page $\Tor^{*}_{H^*(BL)}(H^*(\col BS);\Z)$ and converges to $\Tor^{*}_{C^*(BL)}(C^*(\col BS);\Z)$. However, by formality of $BL$ and of $\lim BS$ (see Theorem \ref{thm:form}), these pages coincide. Hence, by Proposition \ref{pr:qiso1} this spectral sequence collapses at the second page, which proves the first claim. The second claim then follows trivially from Theorem \ref{thm:limcoh}.
\end{proof}

\begin{remark}
Recall that the quotient $(D^d,S^{d-1})/H_d$ is called a \textit{partial quotient} \cite{bu-pa-15} if the corresponding $H_d$-action is free. In the particular case of partial quotients the claim of Theorem \ref{thm:em} was previously known (see Example \ref{ex:freelim}). We refer to \cite{fr-21} for the necessary bibliographical links and for the recent historical overview on the results about cohomology groups and rings of partial quotients. Theorem \ref{thm:em} is a new generalization of previously known results on the above Eilenberg-Moore spectral sequence for partial quotients to the case of any (not necessarily freely acting) closed subgroup $H$ (in $T^m$) satisfying Condition \ref{cond:1}.
\end{remark}

\begin{remark}
Puppe's lemma \cite{fa-96} and cofibrancy of $BS_d$ imply that there is the following $L_d$-equivariant Serre fibration:
\begin{equation}\label{eq:hocserre}
\begin{tikzcd}
\hoc\ Q_d \arrow{r} & \col\ BS_d \arrow{r} & BL_d
\end{tikzcd}.
\end{equation}
For $d=2$, one can deduce that the Serre spectral sequence of the diagram of fibrations \eqref{eq:sfd} evaluated at $I\in\cat K$ as well as of the fibration \eqref{eq:hocserre} collapses in the term $E_3$, compare with \cite[p.115, Proposition 7.36]{bu-pa-02}.
\end{remark}


\section{Cohomology of partial quotients: concluding remarks and open problems}

In this section we are going to discuss the toral rank conjecture and torsion in the integral cohomology of partial quotients. The following conjecture was formulated by Halperin in~\cite{H85}.

\begin{conjecture}\label{halperinconj}
Let $X$ be a finite-dimensional CW complex. Then the inequality holds
$$
\hrk(X):=\sum\limits_{i\geq 0}\dim H^i(X;\Q)\geq 2^{\trk(X)},
$$
where $\trk(X)$ denotes the maximal rank of a torus acting almost freely on $X$.
\end{conjecture}

In what follows we restrict our attention to (almost) free actions of toric subgroups in $T^m$ on $\zk$ and on its partial quotients ($m=f_0(K)$), and we denote by $S^1_D$ the diagonal circle in $T^m$. In this section we mainly study the family of spaces of the type $\zk/S^1_D$.


Recall that the \textit{Buchstaber number} $s(K)$ of a simplicial complex $K$ is the maximal rank of a complex torus acting freely on $\zk$. Following the notation from~\cite{fu-19}, we denote by $\Delta^{k}_m$ the $k$-skeleton of the $(m-1)$-dimensional simplex for $m\geq 2$ and $0\leq k\leq m-2$. The Buchstaber numbers for the spaces in this class were computed in~\cite{fm-11}.

\begin{theorem} 
For any $m\geq 2$ and $0\leq k\leq m-2$, the partial quotient $\mathcal Z_{\Delta^k_m}/S^1_D$ is a rationally formal space with torsion free integral cohomology. Moreover, the following inequality holds:
$$
\hrk(\mathcal Z_{\Delta^k_m}/S^1_D)\geq (k+2)2^{m-k-2}.
$$
\end{theorem}
\begin{proof}
Due to~\cite[Theorem 4.5.10]{fu-19}, for any $0\leq k\leq m-2$ one has:
\[
\mathcal{Z}_{\Delta_{m}^k}/S^1_{D}\simeq \mathbb C P^{k+1}\vee \mathcal{Z}_{\Delta^{k}_{m-1}}\vee (\bigvee_{i=1}^{k} S^{2i-1}\ast \mathcal{Z}_{\Delta_{m-i-1}^{k-i}} )\vee (S^{2k+1}\ast T^{m-k-2}).
\]
Furthermore, by~\cite[Corollary 9.5]{gr-th-07}, we obtain:
\[
\mathcal{Z}_{\Delta_{m}^k}\simeq \bigvee_{j=k+2}^{m} (S^{k+j+1})^{\vee \binom{m}{j}\binom{j-1}{k+1}}.
\]
It follows immediately that the partial quotient $\mathcal Z_{\Delta^k_m}/S^1_D$ is a rationally formal space with torsion free integral cohomology. 

One has the homotopy equivalence $\Sigma T^n\simeq S^2\vee \Sigma T^{n-1}\vee \Sigma^{2}T^{n-1}$ for each $n\geq 2$. It implies that $\Sigma T^n$ is a homotopy wedge of spheres. Therefore the formula
$$
\hrk(\mathcal{Z}_{\Delta_{m}^k}/S^1_{D})=1+(\hrk(\mathbb C P^{k+1})-1)+(\hrk(\mathcal{Z}_{\Delta^{k}_{m-1}})-1)+\sum\limits_{i=1}^{k} (\hrk(\mathcal{Z}_{\Delta_{m-i-1}^{k-i}})-1)+(\hrk(\Sigma T^{m-k-2})-1),
$$
holds.

Note that the formulas $\hrk(\mathbb C P^{k+1})=k+2$ and $\hrk(\Sigma T^{m-k-2})=2^{m-k-2}$ take place. 
Moreover, $\hrk(\mathcal{Z}_{\Delta^{k}_{m-1}})\geq 2^{(m-1)-(k+1)}=2^{m-k-2}$, due to~\cite[Theorem 10]{Ust}, holds.
Hence, it follows that
$$
\hrk(\mathcal{Z}_{\Delta_{m}^k}/S^1_{D})\geq (k+2)-1+2^{m-k-2}+k(2^{m-k-2}-1)+(2^{m-k-2}-1)=(k+2)2^{m-k-2},
$$
takes place. The proof is complete.
\end{proof}

The number $s(P)=s(\partial P^*)$ is called the Buchstaber number of a simple polytope $P$.

\begin{problem}
\begin{itemize}
\item[(a)] Does there exist a simplicial complex $K$ on the vertex set $[m]$ and a toric subgroup $H\subseteq T^m$ of rank $r, 1\leq r\leq s(K)$ acting freely on $\zk$ such that the partial quotient $\zk/H$ is not formal?
\item[(b)] Does there exist a simple polytope $P$ with $m$ facets and a toric subgroup $H\subseteq T^m$ of rank $r, 1\leq r\leq s(P)$ acting freely on $\zp$ such that the partial quotient $\zp/H$ is not formal?
\end{itemize}
\end{problem}

If $f_0(K)=m$ and $\dim K=n-1$, then the maximal rank of a toric subgroup in $T^m$ acting almost freely on $\zk$ equals $m-n$ by~\cite[Lemma 8]{Ust} and~\cite[\S7.1]{da-ja-91}. 

\begin{lemma}\label{lemmatoralrank}
In the above notation, the maximal rank of a toric subgroup in $T^m$ acting almost freely on $\zk/S^1_D$ is less or equal to $m-n-1$.
\end{lemma}
\begin{proof}
Obviously, the stabilizer of a point $x\in (D^2,S^1)^I\subseteq\zk$ with respect to the $T^m$-action is equal to $T^I$, where $I\in K$. 
Therefore, the stabilizer of a point $x\in (D^2,S^1)^I/H\subseteq \zk$ with respect to the $T^m$-action is the subgroup $H(I)$ generated by $T^I$ and $S^1_D$ in $T^m$. The subgroup $H(I)$ does not depend on the choice of the point $x\in (D^2,S^1)^I$.

A toric subgroup $T^r\subseteq T^m$ acts almost freely on $\zk/S^1_D$ if and only if the intersection $T^r\cap H(I)$ is finite for all $I\in K$. In this case we have:
$$
\rk(T^r)+\rk(H(I))\leq\rk(T^m).
$$

The intersection $T^I\cap S^1_D$ is trivial by the condition on the almost free action. Hence, the group
$H(I)$ is the inner direct product of the subgroups $T^I$ and $S^1_D$ in $T^m$. Therefore, one has $\rk(H(I))=n+1$ for any $\dim I=n-1$. It follows that
$$
r=\rk(T^r)\leq \rk(T^m)-\rk(H(I))=m-n-1,
$$
holds. The proof is complete.
\end{proof}

As it was already mentioned above, here we consider only (almost) free actions of toric subgroups in $T^m$ on $\zk$ and on its partial quotients. The weaker version of the Halperin's conjecture for moment-angle complexes was proved in~\cite{Ust}. We show that the similar statement holds for partial quotients of moment-angle complexes by the diagonal circle action.

\begin{theorem}\label{thm:hrkatr}
For any simplicial complex $K$ of dimension $n-1$ on $m$ vertices, one has the inequality
$$
\hrk(X)\geq 2^{\mathrm{atr}(X)},
$$
where $X=\zk/S^1_D$ and $\mathrm{atr}(X)$ denotes the maximal rank of a toric subgroup in $T^m$ acting almost freely on $X$.
\end{theorem}
\begin{proof}
Consider the principal $S^1$-bundle $S^1_D\to\zk\to X$. The space $X$ is simply connected which follows from the homotopy exact sequence of this bundle, because the moment-angle complex $\zk$ is $2$-connected. 

The inequality $\hrk(\zk)\leq\hrk(X)\cdot\hrk(S^1_D)=2\hrk(X)$ follows directly from the Serre spectral sequence of this bundle. Hence, $\hrk(X)\geq\hrk(\zk)/2\geq 2^{m-n-1}$ holds, where the last inequality takes place by~\cite[Theorem 10]{Ust}. 

On the other hand, one has the inequality $\mathrm{atr}(X)\leq m-n-1$ by Lemma~\ref{lemmatoralrank}. This finishes the proof.
\end{proof}

We are grateful to Anton Ayzenberg for drawing our attention to the fact that the original version of the Conjecture~\ref{halperinconj} remains open for both moment-angle complexes and more general partial quotients.

\begin{problem}
Prove Conjecture~\ref{halperinconj} for all partial quotients of moment-angle complexes, or find a counterexample in this class of spaces.
\end{problem}

In what follows we discuss torsion in the integral cohomology of partial quotients.

\begin{theorem}\label{thm:torsion}
Let $G$ be a finitely generated abelian group. Then there exists a simple polytope $P\subseteq\R^n$ with $m$ facets and a toric subgroup of rank one (a circle) $H\subseteq T^m$ such that $H$ acts freely on the moment-angle manifold $\zp$ and $H^*(\zp/H)$ contains the group $G$ as its direct summand. 
\end{theorem}
\begin{proof}
Consider the Moore space $X$ for $G$, that is, $H^{p}(X)\cong G$ and $\tilde{H}^{i}(X)=0$ holds for all $i\neq p$, given a certain $p\geq 1$. Take its arbitrary finite triangulation $K$ and let $K^\prime$ be obtained from $K$ by a stellar subdivision in a maximal simplex of $K$. Then there exists a pair of distinct vertices, $i$ and $j$, in the vertex set $[m]$ of the complex $K^\prime$, not linked by an edge. 

Following~\cite{bo-me-06}, consider the full simplex $\Delta_{[m]}$ on the vertex set $[m]$ and let us cut off its faces, one by one, corresponding to the minimal non-faces of the complex $K^\prime$. Then the nerve complex $\tilde{K}$ of the resulting simple polytope $P$ will be a polytopal sphere of dimension $m-2$ with $M:=m+|MF(K^\prime)|$ vertices, where $MF(K^\prime)$ denotes the set of minimal non-faces of the complex $K^\prime$. Moreover, there still exists no edge connecting the vertices $i$ and $j$ of $\tilde{K}$, because $\{i,j\}\in MF(K^\prime)\subseteq MF(\tilde{K})$ holds. 

Now, consider the partition $\alpha$ of the set $[M]$ into $M-1$ classes $\alpha_1,\ldots,\alpha_{M-1}$, where $\{i,j\}$ is the unique class consisting of $2$ elements. Following the notation from~\cite{li-19}, one has the mapping $\lambda_{\alpha}\colon [M]\to\Z^{M-1}$, which sends a vertex from the class $\alpha_k$ to $\tilde{e}_k$, where $\{\tilde{e}_1,\ldots,\tilde{e}_{M-1}\}$ denotes the basis of the lattice $\Z^{M-1}$. Due to~\cite{li-19}, we observe that the mapping $\lambda_\alpha$ gives rise to the toric subgroup $H_{\lambda_\alpha}\subset T^{M}$, acting freely on the moment-angle manifold $\zp$. 

Therefore, the space $X(\tilde{K},\lambda_\alpha)\cong\zp/H$ will be the corresponding partial quotient. By our construction, it will be a smooth compact simply connected manifold of dimension $\dim\zp-1=2m+|MF(K^\prime)|-2$, and the equality $\tilde{K}_{\alpha,[m-1]}=K^\prime$ holds. Hence, by~\cite[Theorem~1.2]{li-19}, $H^{q}(\zp/H)$ contains $G$ as its direct summand, when $q=p+m$. This finishes the proof.
\end{proof}

\begin{remark}
Let $H^*(\zk)$ be torsion free. Then the integral cohomology rings of the partial quotients are also torsion free for the moment-angle complexes $\zk$ from the class introduced in~\cite{li-19}: this follows immediately from~\cite[Theorem 1.2]{li-19}.
\end{remark}

\begin{example}
Let $G=\Z/2\Z$. Take $X$ to be $\R P^2$. Consider its minimal triangulation on $6$ vertices: $K=\R P^2_6$. It is well-known that the set of the minimal non-faces $MF(K)$ consists of ten $3$-element sets. Hence, the $2$-dimensional simplicial complex $K^\prime$ has $m=7$ vertices and the $5$-dimensional polytopal sphere $\tilde{K}$ has $M=21$ vertices. Therefore, $\dim \zp/H=26$ holds and its integral cohomology contains $2$-torsion in degree $q=9$.
\end{example}

Note that in general the orbit space $\zp/T^r$ has torsion in integral homology provided that the $T^r$-action is not free. See~\cite[Example 2.4]{fi-92}.

\begin{problem}
\begin{itemize}
\item[(a)] Does there exist a simple polytope $P$ such that $H^*(\zp;\Z)$ has torsion and the groups $H^*(\zp/T^r;\Z)$ are free for any freely acting toric subgroup $T^r\subset T^m$, $1\leq r\leq s(P)$?
\item[(b)] Does there exist a simplicial complex $K$ such that $H^*(\zk;\Z)$ has torsion and the groups $H^*(\zk/T^r;\Z)$ are free for any freely acting toric subgroup $T^r\subset T^m$, $1\leq r\leq s(K)$?
\end{itemize}
\end{problem}

\begin{problem}
\begin{itemize}
\item[(a)] Does there exist a simple polytope $P$ such that $H^*(\zp;\Z)$ is a free group and $H^*(\zp/T^r;\Z)$ has torsion for a certain freely acting toric subgroup $T^r\subset T^m$, $1\leq r\leq s(P)$?
\item[(b)] Does there exist a simplicial complex $K$ such that $H^*(\zk;\Z)$ is a free group and $H^*(\zk/T^r;\Z)$ has torsion for a certain freely acting toric subgroup $T^r\subset T^m$, $1\leq r\leq s(K)$?
\end{itemize}
\end{problem}


{\emph{Acknowledgements.}} We are grateful to Anton Ayzenberg, Victor Buchstaber, Fyodor Pavutnitsky and Taras Panov for a number of fruitful discussions. The authors are also obliged to Anton Ayzenberg for attracting their attention to the work~\cite{li-19} related to the problem of the existence of nontrivial torsion in integral cohomology of partial quotients. The authors wish to thank Anthony Bahri and Matthias Franz for their interest in this work and for valuable comments. Finally, we are deeply grateful to the anonymous referee for the proposed generalization of the weak version of the Halperin's conjecture from the case when $K$ is a skeleton of a simplex to the case of an arbitrary simplicial complex $K$ (Theorem ~\ref{thm:hrkatr}), for pointing out further possible directions of research related to the Theorem ~\ref{thm:torsion}, as well as for the numerous useful comments that contributed to the improvement of the paper.

The article was prepared within the framework of the Basic Research Program at HSE University, RF. The first author is also a Young Russian Mathematics award winner and would like to thank its sponsors and jury.


\end{document}